\documentclass[12pt]{amsart}

\usepackage[letterpaper,margin=.6in]{geometry}
\usepackage[colorlinks,linkcolor=blue,citecolor=black!75!red]{hyperref}

\usepackage{url}
\usepackage{color}
\usepackage[all]{xypic}

\usepackage{latexsym}
\usepackage{amssymb}
\usepackage{amsfonts}
\usepackage{amscd}
\usepackage{amsmath,amsthm}
\usepackage{verbatim}

\usepackage{tikz}
\usetikzlibrary{matrix,arrows,decorations.pathmorphing,fit,cd}
\tikzset{myboxgroup/.style={draw, densely dotted}} 

\usepackage{cleveref}

\newtheorem{lemma}{Lemma}[section]
\newtheorem{proposition}[lemma]{Proposition}
\newtheorem{theorem}[lemma]{Theorem}
\newtheorem{corollary}[lemma]{Corollary}

{

}

\theoremstyle{definition}

\theoremstyle{remark}

\makeatletter
\let\xx@thm\@thm
\AtBeginDocument{\let\@thm\xx@thm}
\makeatother

\crefname{section}{Section}{Sections}
\crefformat{section}{#2section~#1#3}
\Crefformat{section}{#2Section~#1#3}

\crefformat{subsection}{#2\S#1#3}
\Crefformat{subsection}{#2\S#1#3}
\crefrangeformat{subsection}{\S\S#3#1#4--#5#2#6}
\Crefrangeformat{subsection}{\S\S#3#1#4--#5#2#6}
\crefmultiformat{subsection}{\S\S#2#1#3}{ and~#2#1#3}{, #2#1#3}{ and~#2#1#3}
\Crefmultiformat{subsection}{\S\S#2#1#3}{ and~#2#1#3}{, #2#1#3}{ and~#2#1#3}
\crefrangemultiformat{subsection}{\S\S#3#1#4--#5#2#6}{ and~#3#1#4--#5#2#6}{, #3#1#4--#5#2#6}{ and~#3#1#4--#5#2#6}
\Crefrangemultiformat{subsection}{\S\S#3#1#4--#5#2#6}{ and~#3#1#4--#5#2#6}{, #3#1#4--#5#2#6}{ and~#3#1#4--#5#2#6}

\crefname{definition}{Definition}{Definitions}
\crefformat{definition}{#2Definition~#1#3}
\Crefformat{definition}{#2Definition~#1#3}

\crefname{definitionnodiamond}{Definition}{Definitions}
\crefformat{definitionnodiamond}{#2Definition~#1#3}
\Crefformat{definitionnodiamond}{#2Definition~#1#3}

\crefname{example}{Example}{Examples}
\crefformat{example}{#2Example~#1#3}
\Crefformat{example}{#2Example~#1#3}

\crefname{examplenodiamond}{Example}{Examples}
\crefformat{examplenodiamond}{#2Example~#1#3}
\Crefformat{examplenodiamond}{#2Example~#1#3}

\crefname{remark}{Remark}{Remarks}
\crefformat{remark}{#2Remark~#1#3}
\Crefformat{remark}{#2Remark~#1#3}

\crefname{remarknodiamond}{Remark}{Remarks}
\crefformat{remarknodiamond}{#2Remark~#1#3}
\Crefformat{remarknodiamond}{#2Remark~#1#3}

\crefname{convention}{Convention}{Conventions}
\crefformat{convention}{#2Convention~#1#3}
\Crefformat{convention}{#2Convention~#1#3}

\crefname{notation}{Notation}{Notations}
\crefformat{notation}{#2Notation~#1#3}
\Crefformat{notation}{#2Notation~#1#3}

\crefname{notationnodiamond}{Notation}{Notations}
\crefformat{notationnodiamond}{#2Notation~#1#3}
\Crefformat{notationnodiamond}{#2Notation~#1#3}

\crefname{lemma}{Lemma}{Lemmas}
\crefformat{lemma}{#2Lemma~#1#3}
\Crefformat{lemma}{#2Lemma~#1#3}

\crefname{proposition}{Proposition}{Propositions}
\crefformat{proposition}{#2Proposition~#1#3}
\Crefformat{proposition}{#2Proposition~#1#3}

\crefname{corollary}{Corollary}{Corollaries}
\crefformat{corollary}{#2Corollary~#1#3}
\Crefformat{corollary}{#2Corollary~#1#3}

\crefname{theorem}{Theorem}{Theorems}
\crefformat{theorem}{#2Theorem~#1#3}
\Crefformat{theorem}{#2Theorem~#1#3}

\crefname{assumption}{Assumption}{Assumptions}
\crefformat{assumption}{#2Assumption~#1#3}
\Crefformat{assumption}{#2Assumption~#1#3}

\crefname{enumi}{}{}
\crefformat{enumi}{(#2#1#3)}
\Crefformat{enumi}{(#2#1#3)}

\crefname{equation}{}{}
\crefformat{equation}{(#2#1#3)}
\Crefformat{equation}{(#2#1#3)}

\crefname{align}{}{}
\crefformat{align}{(#2#1#3)}
\Crefformat{align}{(#2#1#3)}

\crefname{proofstep}{Step}{Steps}
\crefformat{proofstep}{#2Step~#1#3}
\Crefformat{proofstep}{#2Step~#1#3}

\crefname{table}{Table}{Tables}
\crefformat{table}{#2Table~#1#3}
\Crefformat{table}{#2Table~#1#3}

\numberwithin{equation}{section}

\def\CC{{\mathbb C}}

\def\FF{{\mathbb F}} 

\def\HH{{\mathbb H}}

\def\PP{{\mathbb P}}

\def\RR{{\mathbb R}}

\def\ZZ{{\mathbb Z}}

\newcommand{\bbZ}{\mathbb{Z}}
\newcommand{\bbC}{\mathbb{C}}

\def\bC{{\mathbb C}}
\def\bH{{\mathbb H}}
\def\bQ{{\mathbb Q}}
\def\bR{{\mathbb R}}
\def\bZ{{\mathbb Z}}

\def\0ol{{\bar 0}}
\def\1ol{{\bar 1}}
\def\2ol{{\bar 2}}
\def\ol2{{\bar 2}}
\def\3ol{{\bar 3}}
\def\4ol{{\bar 4}}
\def\5ol{{\bar 5}}
\def\6ol{{\bar 6}}
\def\7ol{{\bar 7}}
\def\8ol{{\bar 8}}
\def\9ol{{\bar 9}}

\def\bold0{{\bf 0}}
\def\bold1{{\bf 1}}
\def\bold2{{\bf 2}} 
\def\bold3{{\bf  3}}
\def\bold4{{\bf 4}}
\def\bold5{{\bf 5}}
\def\bold6{{\bf 6}}
\def\bold7{{\bf 7}}
\def\bold8{{\bf 8}}
\def\bold9{{\bf 9}}

\def\P2Skly{\PP^2_{Skly}}

\def\a{\alpha}

\def\l{\lambda}

\def\s{\sigma}

\def\G{\Gamma}
\def\L{\Lambda}

\def\cal{\mathcal}

\def\cL{{\cal L}}

\def\cO{{\cal O}}

\def\End{\operatorname {End}}

\def\Aut{\operatorname{Aut}}

\def\id{\operatorname{id}}

\def\th{{\rm th}}

\def\dirlim{\mathop{\vtop{\baselineskip -100pt\lineskip -1pt\lineskiplimit 0pt
\setbox0\hbox{lim}\copy0\hbox to \wd0{\rightarrowfill}}}\limits}
\def\invlim{\mathop{\vtop{\baselineskip -100pt\lineskip -1pt\lineskiplimit 0pt
\setbox0\hbox{lim}\copy0\hbox to \wd0{\leftarrowfill}}}\limits}

\def\I11{{1 \kern -0.8pt \! \mbox{l}}}
\def\mumu{{\mu\kern-4.2pt\mu}}
\def\bfmu{{\mu\kern-4.2pt\mu}}
\def\2slash{\backslash \! \backslash}

\def\GL{\mathrm{GL}}

\def\SL{\mathrm{SL}}

\makeatletter
\def\l@subsection{\@tocline{2}{0pt}{2.75pc}{5pc}{}}
\makeatother

\makeatletter
\@namedef{subjclassname@2020}{\textup{2020} Mathematics Subject Classification}
\makeatother

\begin{document}

\title[Modular properties of elliptic algebras]{Modular properties of elliptic algebras}

\author{Alex Chirvasitu, Ryo Kanda, and S. Paul Smith}

\address[Alex Chirvasitu]{Department of Mathematics, University at
  Buffalo, Buffalo, NY 14260-2900, USA.}
\email{achirvas@buffalo.edu}

\address[Ryo Kanda]{Department of Mathematics, Graduate School of Science, Osaka City University, 3-3-138, Sugimoto, Sumiyoshi, Osaka, 558-8585, Japan.}
\email{ryo.kanda.math@gmail.com}

\address[S. Paul Smith]{Department of Mathematics, Box 354350,
  University of Washington, Seattle, WA 98195, USA.}
\email{smith@math.washington.edu}
 
\keywords{Elliptic algebras; elliptic curves;  theta functions;  modular group; finite Heisenberg group; automorphisms of the Heisenberg group; modular functions; Yang-Baxter equation}

\subjclass[2020]{16S38 (Primary), 16S37, 16W50, 14H52 (Secondary)}

\begin{abstract}
Fix a pair of relatively prime integers $n>k\ge 1$, and a point $(\eta\,|\,\tau)\in\mathbb{C}\times\mathbb{H}$, where $\mathbb{H}$ denotes the upper-half complex plane, and let ${{a\;\,b}\choose{c\,\;d}}\in\mathrm{SL}(2,\mathbb{Z})$. We show that Feigin and Odesskii's elliptic algebras $Q_{n,k}(\eta\,|\,\tau)$ have the property $Q_{n,k}\big(\frac{\eta}{c\tau+d}\,\big\vert\,\frac{a\tau+b}{c\tau+d}\big)\cong Q_{n,k}(\eta\,|\,\tau)$. As a consequence, given a pair $(E,\xi)$ consisting of a complex elliptic curve $E$ and a point $\xi\in E$, one may unambiguously define $Q_{n,k}(E,\xi):=Q_{n,k}(\eta\,|\,\tau)$ where $\tau\in\mathbb{H}$ is any point such that $\mathbb{C}/\mathbb{Z}+\mathbb{Z}\tau\cong E$ and $\eta\in\mathbb{C}$ is any point whose image in $E$ is $\xi$. This justifies Feigin and Odesskii's notation $Q_{n,k}(E,\xi)$ for their algebras.
\end{abstract}

\maketitle

\tableofcontents

\section{Introduction}
\label{sect.defn.Qnk} 

We will use the notation $e(z):=e^{2\pi i z}$.

Always, $\tau$ denotes a point in the upper-half complex plane, $\HH:=\{x+iy \, | \, x,y \in \RR, \; y>0\}$.

\subsection{The main result}
Let 
\begin{equation}
\label{defn.vartheta} 
\vartheta(z \, | \, \tau)  \; := \; \sum_{n\in \ZZ} e\big(nz + \tfrac{1}{2} n^2\tau\big).
\end{equation}
In 1828, C.G.J. Jacobi \cite{jac-1828} proved the remarkable identity
\begin{equation}
\label{defn.Jac.id}
\vartheta \bigg( \frac{z}{\tau} \, \Big\vert\,- \frac{1}{\tau} \bigg) \; =\; \sqrt{-i\tau }\, e\bigg(\frac{z^{2}}{2\tau}\bigg)\vartheta(z \, | \, \tau),
\end{equation} 
where $\sqrt{-i\tau}$ is the square root of $-i\tau$ having non-negative real part.\footnote{Jacobi did not use this notation, nor did he provide a proof in \cite{jac-1828}. Jacobi's identity is equivalent to a similar identity for the function 
$\vartheta_1(z \, | \, \tau)$ that appears in Chapter XXI of Whittaker-Watson \cite[pp.~475--476]{WW-4th-ed}.
On the other hand, \Cref{defn.Jac.id} is equivalent to a similar identity for Whittaker and Watson's function 
$\vartheta_3(z \, | \, \tau)$ since $\vartheta(z \, | \, \tau)=\vartheta_3(\pi z \, | \tau)$ and, as Whittaker and Watson
indicate, their identities for $\vartheta_1(z \, | \, \tau)$ and $\vartheta_3(z \, | \, \tau)$ are equivalent.}
Jacobi's identity is a special case of the following functional equation: if 
\begin{equation}
\label{eq:SL2Z.elt}
\begin{pmatrix}
  a    &   b \\
   c   &  d
\end{pmatrix}  
\;  \in \; \SL(2,\ZZ)
 \end{equation}
 is such that $ab$ and $cd$ are even, then there is an $8^{\rm th}$ root of unity $\zeta$ such that
\begin{equation}
\label{eq.modular.identity}
\vartheta \bigg( \frac{z}{c\tau+d} \;  \bigg\vert \;  \frac{a\tau+b}{c\tau+d}  \bigg)  
\;=\; \zeta \sqrt{c\tau+d}  \, \, e\bigg(\frac{cz^2}{2(c\tau+d)} \bigg) \, \vartheta(z \, | \, \tau).
 \end{equation}
The precise value of $\zeta$, which is not important in this paper, can be found at
\cite[Thm.~7.1, p.~32]{Mum07}, for example.

In 1989, Feigin and Odesskii defined a remarkable family of graded associative $\CC$-algebras which we denote here by 
$Q_{n,k}(\eta \, | \, \tau)$. These are the elliptic algebras of the title. 
They depend on a  pair of relatively prime integers $n>k \ge 1$ and a point 
$(\eta \, | \, \tau) \in \CC \times \HH$. We define $Q_{n,k}(\eta \, | \, \tau)$ in \Cref{ssect.defn}.

The main result in this paper is the following.

\begin{theorem}[\Cref{th:qequiv}]
\label{thm.main}
If (\ref{eq:SL2Z.elt}) holds, then there is an isomorphism of graded $\CC$-algebras
\begin{equation}
\label{eq:main.isom}
Q_{n,k} \bigg( \frac{\eta}{c\tau+d} \; \bigg\vert \;  \frac{a\tau+b}{c\tau+d}  \bigg) \; \cong \; Q_{n,k}(\eta \, | \, \tau).
\end{equation}
\end{theorem}
 
In \cref{subse:mod} we explain how this result can be interpreted 
as saying that $Q_{n,k}(\, \cdot \, | \, \cdot \, )$ is a ``weakly modular function  of 
weight $-1$ taking values in the category of graded $\CC$-algebras''. 
 
\subsection{The definition of $Q_{n,k}(\eta \, | \, \tau)$}
\label{ssect.defn}
Fix the data $(n,k,\eta,\tau)$ as above; i.e.,  $n > k \ge 1$ is a pair of relatively prime positive integers and 
$(\eta \, | \, \tau)$ is a point in $\CC \times \HH$.

Write $\Lambda_\tau:= \bZ+\bZ\tau$ and $E_\tau:=\CC/\L_\tau$. 

Fix a complex vector space $V$ with basis $x_0,\ldots,x_{n-1}$ indexed by $\bZ_n$, the 
cyclic group of order $n$. 
In \cite[Prop.~2.6]{CKS1}, we defined a particular basis $\theta_0,\ldots,\theta_{n-1}$, also  indexed by $\bZ_n$,  
 for the vector space $\Theta_n(\Lambda_\tau)$ which, by definition, consists of all holomorphic functions $f:\bC \to \bC$ such that
\begin{align*} 
f(z+1) & \;=\; f(z) \qquad  \qquad \text{$\quad$ and}
\\
f(z+\tau) & \; = \; - e(- n z) f(z)
\end{align*}
for all $z \in \bC$. Copying Feigin and Odesskii, we then defined for each 
$\eta \in \bC-\frac{1}{n}\Lambda_\tau$ a family of holomorphic linear operators
\begin{equation*}
R_{n,k}(z,\eta \, | \, \tau):V \otimes V \; \longrightarrow \; V \otimes V, \qquad z \in \bC,
\end{equation*}
given by the formula
\begin{equation}\label{the.relns}
  R_{n,k}(z,\eta\,|\,\tau)(x_i\otimes x_j)  \; := \; 
  \frac{\theta_0(-z) \cdots \theta_{n-1}(-z)}{\theta_1(0) \cdots \theta_{n-1}(0)}
  \sum_{r\in \bZ_n}
  \frac{\theta_{j-i+r(k-1)}(-z+\eta)}
  {\theta_{j-i-r}(-z)\theta_{kr}(\eta)} \,
  x_{j-r}\otimes x_{i+r}
\end{equation}
for all $(i,j) \in \bZ_n^2$.  Although this formula does not make sense when $\eta \in \frac{1}{n}\Lambda$, because in that case $\theta_{kr}(\eta)=0$ for some $r \in \ZZ_n$, the  function $R_{n,k}(\eta,\eta\,|\,\tau):\bC-\frac{1}{n}\Lambda_\tau \, \longrightarrow \, \End_\bC(V^{\otimes 2})$ extends in a unique way to a holomorphic function on the entire complex plane \cite[Lem.~3.13]{CKS1}---we will write $R_{n,k}(\eta,\eta\,|\,\tau)$ for the extension. This allows us to define, for all 
$\eta \in \bC$, the algebra
\begin{equation}
\label{defn.Q.eta.tau}
Q_{n,k}(\eta \, | \, \tau) \; :=\; \frac{TV}{(\text{the image of } R_{n,k}(\eta,\eta\,|\,\tau))} \, 
\end{equation}
where $TV$ denotes the tensor algebra on $V$.   

Although the formula for $R_{n,k}(z,\eta\,|\,\tau)$ is not very illuminating, by building on results and ideas of Richey and Tracy \cite{ric-tra} we 
were able to show in \cite{CKS4} that $R(z)=R_{n,k}(z,\eta\,|\,\tau)$ 
satisfies the quantum Yang-Baxter equation: for all $u,v \in \bC$,
\begin{equation*}
R(u)_{12} R(u+v)_{23}R(v)_{12} \;=\;  R(v)_{23} R(u+v)_{12}R(u)_{23} 
\end{equation*}
as operators on $V^{\otimes 3}$, where $R(z)_{12}:=R(z) \otimes I$, $R(z)_{23}:=I \otimes R(z)$, and $I$ is the identity operator on $V$. 
This identity plays a key role in the proof of several results in \cite{CKS4} which showed,  for suitable $\eta$, that $Q_{n,k}(\eta \, | \, \tau)$ has many of the good properties enjoyed by the polynomial ring on $n$ variables; for example, when the image of $\eta$ in 
$E_\tau$ is not a torsion point, $Q_{n,k}(\eta \, | \, \tau)$ has the same Hilbert series as that polynomial ring. 

The work of Feigin and Odesskii in \cite{FO-Kiev, FO89, OF95, Od-survey}, and our later work in \cite{CKS1,CKS2,CKS3,CKS4}, 
shows that $Q_{n,k}(\eta \, | \,\tau)$ has many beautiful properties. Roughly speaking, the algebras $Q_{n,k}(\eta \, | \,\tau)$ are a step beyond enveloping algebras and quantized enveloping algebras in the same way as elliptic and theta functions are a step beyond
rational and trigonometric functions.

\subsection{Feigin and Odesskii's elliptic algebras $Q_{n,k}(E,\xi)$}
\label{ssect.notation}

Let $X$ denote a compact Riemann surface of genus one. A pair $(X,p)$ consisting of a compact Riemann surface $X$ of genus one 
and a point $p$ on it will be called a {\sf complex elliptic curve}. We usually  omit the adjective ``complex''. 
There is a unique way to make $X$ an algebraic (or Lie) group 
with $p$ as its identity---see, e.g., \cite[Cor.~1.11]{Hain11}. We will always view $(X,p)$ as a group in this way. 

Given an elliptic  curve $E=(X,p)$ and a point $\xi \in E$, i.e., a second point $\xi \in X$, which is allowed to be $p$,
Feigin and Odesskii defined an algebra they denoted by $Q_{n,k}(E,\xi)$---see \cite{FO89}, and \cite{CKS1}
for a definition when $\xi$ is one of the points that Feigin and Odesskii disallow.\footnote{In truth, Feigin and Odesskii use the notation $Q_{n,k}(E,\tau)$ where $\tau$ denotes a point on $E$---see \Cref{ssect.tau.eta,ssect.warning} below.}

At \cite[\S1.2]{FO89}, Feigin and Odesskii define $Q_{n,k}(E,\xi)$ by the formula on the right-hand side of \Cref{defn.Q.eta.tau}
where $(\eta \, | \, \tau) \in \CC \times \HH$ is such that $E$ is isomorphic to 
$E_\tau$ and $\eta$ is a point whose image in $\CC/\ZZ+\ZZ\tau$ is $\xi$. 

There are many points $(\eta \, | \, \tau) \in \CC \times \HH$ with the property in the previous sentence.
It is easy to see that $Q_{n,k}(\eta \, | \, \tau) = Q_{n,k}(\eta' \, | \, \tau)$ if $\eta$ and $\eta'$ have the same image in 
$E_\tau$ so Feigin and Odesskii's definition of $Q_{n,k}(E,\xi)$ does not depend on
the choice of $\eta$. However, Feigin and Odesskii do not show that the algebra they denote by 
$Q_{n,k}(E,\xi)$ stays the same when $\tau$ is replaced by another  point $\tau' \in \HH$ for which $E$ is isomorphic to 
$E_{\tau'}$. We also failed to address this issue in our earlier papers  \cite{CKS1,CKS2,CKS3,CKS4}. 
\Cref{thm.main} remedies this oversight and, as we will now explain, allows us to conclude that $Q_{n,k}(E,\xi)$ is well-defined up to isomorphism when we define it as
\begin{equation*}
Q_{n,k}(E,\xi) \; :=\; Q_{n,k}(\eta \, | \, \tau)
\end{equation*}
where $\tau \in \HH$ is any point such that $E_\tau\cong E$ and the image of $\eta$ under this isomorphism is $\xi$.

There is a notion of isomorphism between such pairs $(E,\xi)=(X,p,\xi)$, and it follows from \Cref{thm.main}  
that $Q_{n,k}(E,\xi)$  is isomorphic to $Q_{n,k}(E',\xi')$ when $(E,\xi)$ is isomorphic to $(E',\xi')$.
More precisely:

\begin{theorem}[\Cref{cor.qnk.not}]
\label{thm.main.2}
If $(E,\xi)=(E_{\tau},\eta+\Lambda_{\tau})$ and $(E',\xi')=(E_{\tau'},\eta'+\Lambda_{\tau'})$  and $\mu:E \to E'$ is an isomorphism of algebraic groups such that the diagram
\begin{equation*}
\xymatrix{
E \ar[r]^\mu \ar[d]_{x  \mapsto x +\xi} & E' \ar[d]^{x' \mapsto x'+\xi'}
\\
E \ar[r]_\mu & E' 
}
\end{equation*} 
commutes, then 
\begin{equation*}
Q_{n,k}(\eta\,|\,\tau) \;\cong\; Q_{n,k}(\eta'\,|\,\tau').
\end{equation*}
\end{theorem}

All the results about $Q_{n,k}(E,\xi)$ in Feigin and Odesskii's papers and in our earlier papers are, in fact, results about 
$Q_{n,k}(\eta \, | \, \tau)$.

\subsubsection{Changing the roles of  the symbols $\tau$ and $\eta$}
\label{ssect.tau.eta}
In our papers \cite{CKS1,CKS2,CKS3,CKS4}  we used $\eta$ to denote a point in $\HH$ and $\tau$ to denote a point in $\CC$.
In this paper we switch that notation in order to agree with the common convention in the literature on elliptic functions, theta functions, and modular forms, that $\tau$ denotes a point in $\HH$.

\subsubsection{Warning}
\label{ssect.warning}
In \cite{FO-Kiev,FO89},  Feigin and Odesskii use the symbol $\tau$ for both a point in $\bC$ and a point on $E$---in \cite[\S1.1]{FO89} the symbol $\tau$ denotes a point on $E$ but in \cite[\S1.2]{FO89} the notation $\theta_\alpha(\tau)$ only makes sense when $\tau \in \bC$.

\subsection{The case when $k=1$}
When $k=1$,   \Cref{thm.main.2} was proved by Tate and Van den Bergh \cite{TvdB96}.  
In fact, they proved more because they work with elliptic curves over arbitrary fields. 

In \cite[\S4.1]{TvdB96}, Tate and Van den Bergh define a graded $\FF$-algebra $A(E,\s,\cL)$ when $E$ is an elliptic curve over an 
arbitrary field $\FF$, $\s$ is a translation automorphism of $E$, 
and $\cL$ is an invertible $\cO_E$-module of degree $n \ge 3$. 
They then prove in Proposition 4.1.1 the following result: if $\mu:E' \to E$ is an isomorphism between elliptic curves, then there is a canonical isomorphism 
$A(E,\s,\cL) \to A(E',\mu^{-1}\s\mu,\mu^*\cL)$ sending $x \in A(E,\s,\cL)_1=H^0(E,\cL)$ to $\mu^*x \in H^0(E',\mu^*\cL)$. 

When $E$ is a complex elliptic curve, $\s$ is translation by $\xi$, and $\deg\cL=n$, then 
\begin{equation*}
A(E,\s,\cL) \; \cong \; Q_{n,1}(E,\xi) \; =\; Q_{n,1}(\eta \, | \, \tau)
\end{equation*}
where $\tau \in \HH$ is any point such that $E \cong E_\tau$ and $\xi$ is the image of $\eta$ under an isomorphism $E \to E_\tau$.\footnote{A proof that $A(E,\s,\cL)$ is isomorphic to  $Q_{n,1}(\eta \, | \, \tau)$  
is given in \cite[\S3.2.6]{CKS1}.} Thus, \cite[Prop.~4.1.1]{TvdB96} proves \Cref{thm.main.2} when $k=1$. 

\subsection{Remarks on the proof of the main theorem}
The main theorem is an immediate consequence of the ``functional equation'' \cref{funcl.eq.for.R} below.
Let us explain.

Let $M = {{a \; \,b} \choose {c \, \; d}} \in \SL(2,\ZZ)$ and let $(\eta' \, | \, \tau')=M \triangleright (\eta \, | \, \tau):=\big(\frac{\eta}{c\tau+d} \, | \, \frac{a\tau+b}{c\tau+d}\big)$. 
Since the space of defining relations for $Q_{n,k}(\eta \, | \, \tau)$ is the image of $R_{n,k}(\eta,\eta\,|\,\tau):V^{\otimes 2} \to V^{\otimes 2}$,
to show that $Q_{n,k}(\eta' \, | \, \tau')$ is isomorphic to 
$Q_{n,k}(\eta \, | \, \tau)$ we must show there is an automorphism $\psi \in \GL(V)$, depending on $M$, 
such that $\psi^{\otimes 2}$ sends the image of $R_{n,k}(\eta,\eta\,|\,\tau)$ to the image of $R_{n,k}(\eta',\eta'\,|\,\tau')$; i.e., 
such that 
\begin{equation*}
(\psi^{\otimes 2}) \circ R_{n,k}(\eta, \eta \, | \, \tau) = R_{n,k}(\eta', \eta' \, | \, \tau') \circ \phi
\end{equation*}
for some $\phi \in \GL(V^{\otimes 2})$;  \cref{funcl.eq.for.R} is a stronger and more precise result.

To state \cref{funcl.eq.for.R}  we  extend the action of $\SL(2,\ZZ)$ on $(\eta \, | \, \tau) \in \CC \times \HH$ to an action on
$(z, \eta \, | \, \tau) \in \CC \times \CC \times \HH$, 
\begin{equation*}
M \triangleright (z, \eta \, | \, \tau) \; :=\;  \left(\frac{z}{c\tau+d}\,, \, \frac{\eta}{c\tau+d} \; \bigg| \; \frac{a\tau+b}{c\tau+d}\right).
\end{equation*}
\Cref{th:requiv} shows there is a $\psi \in \GL(V)$ and a nowhere-vanishing holomorphic function $f(z)$, both depending on 
$M \in \SL(2,\ZZ)$, such that 
\begin{equation}
\label{funcl.eq.for.R} 
R_{n,k}\big(M \triangleright (z,\eta \, | \, \tau) \big) \; = \; 
  f(z) \,  \big(\psi^{\otimes 2} \big) \circ R_{n,k}(z,\eta \, | \, \tau) \circ  \big( \psi^{\otimes 2}\big)^{-1}.
\end{equation}

\subsection{The contents of this paper}
The equality \cref{funcl.eq.for.R}  and the main theorem are proved in \cref{se:main}. 
\Cref{se.prel-tht}, which concerns certain normalizations $w_{(u,v)}(z,\eta \, | \, \tau)$ of theta functions with characteristics,
and \cref{sect.Hn}, which is about a finite Heisenberg group $H_n$ of order $n^3$ and its extension $\widetilde{H}_n$ when $n$ is even, are essential preliminaries. 

The key to proving \cref{funcl.eq.for.R}  is to use a presentation of the $R$-matrix defined in \Cref{the.relns}
that has more structure. 
This is done in several steps. 

In \cref{sect.Hn}, we make $V$, the degree-one component of $Q_{n,k}(\eta \, | \tau)$, into an irreducible representation
of $H_n$; as a consequence there is a surjective homomorphism $\rho:\CC H_n \to \End_\CC(V)$
so we can, and do, choose a basis for $\End_\CC(V)$ consisting of images of certain elements $J_{(u,v)} \in H_n$---this works as 
stated when $n$ is odd but needs a small adjustment when $n$ is even; 
we then write $R_{n,k}(z,\eta \, | \tau)$ in terms of this new basis and the functions $w_{(u,v)}(z,\eta \, | \, \tau)$---the essential part of this rewriting involves the operator
$T_k(z,\eta \, | \, \tau):V^{\otimes 2} \to V^{\otimes 2}$ which is defined in \cref{defn.T(-nz)}---its relation to $R_{n,k}(z,\eta \, | \tau)$
is given in \cref{eq:tisr}.

The ``matrix coefficients'' for $T_k(z,\eta \, | \, \tau)$ involve the functions  $w_{(u,v)}(z, \eta \, | \, \tau)$ indexed by $(u,v) \in \ZZ^2$,
which  are introduced in \cref{se.prel-tht}.  The main result in  \cref{se.prel-tht} is that for each $M \in \SL(2,\ZZ)$ there is a 
nowhere-vanishing holomorphic function $g(z)$, depending on $M$ but not on $(u,v)$, such that
  \begin{equation*}
w_{(u,v)}\big(M\triangleright (z,\eta \, |  \, \tau)\big)   \; = \;     g(z) \,  w_{(u,v)M}(z,\eta \, |  \, \tau). 
  \end{equation*}
Thus, as far as $w_{(u,v)}(z,\eta \, | \, \tau)$ is concerned, the modular action of $\SL(2,\ZZ)$ on 
$ \CC \times \CC \times \HH$ can be replaced by the natural right action of  $\SL(2,\ZZ)$ on $\ZZ^2$ or $\ZZ_n^2$.

\subsection*{Acknowledgements} 

A.C. was partially supported by NSF grant DMS-2001128.

R.K. was supported by JSPS KAKENHI Grant Numbers JP16H06337, JP17K14164, JP20K14288, and JP21H04994, Leading Initiative for Excellent Young Researchers, MEXT, Japan, and Osaka City University Advanced Mathematical Institute (MEXT Joint Usage/Research Center on Mathematics and Theoretical Physics JPMXP0619217849).


\section{Theta functions with characteristics and the functions $w_{(u,v)}(z)$}\label{se.prel-tht}

We will use the notation $\omega := e\big(\tfrac{1}{n}\big)$. 

When $\tau \in \HH$ we will write $\Lambda_\tau:=\ZZ+\ZZ\tau$ and $E_\tau:=\CC/\L_\tau$. 

We write $M^t$ for the transpose of a matrix $M$ and $M^{-t}$ for the inverse of $M^t$ when it exists.

\subsection{Definition of $\theta_{u,v}$ }
\label{subse:prel-theta}

As in  \cite[\S2.5]{CKS4}, we will use theta functions with characteristics, namely the functions
\begin{equation*}
  \theta_{u,v}(z \, | \, \tau) \;= \;  \vartheta \! \left[ u \atop v \right] \! (z \, | \, \tau) \; :=\; e\left(u(z+v)  + \tfrac{1}{2}u^2\tau  \right) \, \vartheta(z+u\tau+v \, |  \, \tau)
\end{equation*}
where $\vartheta(z \, |  \, \tau)$ is the function defined in \Cref{defn.vartheta}. In particular,  $\theta_{0,0}=\vartheta$.

\begin{proposition}
\label{prop.theta.uv}\leavevmode
Let $s,t \in \ZZ$.
\begin{enumerate}
  \item\label{eq.theta.v}
$   \theta_{u,v}(z+s \, | \, \tau ) = \theta_{u,v+s}(z \, | \, \tau)$.
  \item\label{eq.theta.st}
 $ \theta_{u,v}(z+s\tau+t \, | \, \tau) = e\big(-s(z+v) -\frac{s^2\tau}2+tu\big) \,  \theta_{u,v}(z \, | \, \tau)$.
  \item\label{eq.theta.zero}
  $\theta_{u,v}(z \, | \, \tau)=0$ if and only if  $z \in \frac 12(\tau+1) - (u\tau+v)+ \Lambda_{\tau}$.
  \item
  Each of these zeros is a simple zero.
\end{enumerate}
\end{proposition}
\begin{proof}
Parts \cref{eq.theta.v}, \cref{eq.theta.st}, and \cref{eq.theta.zero} were stated in \cite[the introductory remarks in \S2.5, and Prop.~2.8]{CKS4},
with the roles of $\eta$ and $\tau$ interchanged. For \cref{eq.theta.zero}, see also \cite[p.~315, (iv)]{ric-tra} and  the discussion before Proposition 2.8 in \cite{CKS4}. 
By \cite[Lem.~2.5]{CKS4}, $\theta_{u,v}(z \, | \, \tau)$ has a single zero in each fundamental parallelogram for $\L_\tau$
located at $\frac 12(1+ \tau) - (u\tau+v)$ modulo $\L_\tau$.  
\end{proof}

\subsection{Definition of $w_{(u,v)}$}\label{subse:def-w}
 
As in \cite[\S3]{CKS4}, we define
\begin{equation}\label{eq:6}
  w_{(u,v)}(z) \;:=\; \frac{\theta_{\frac un,\frac vn}(z+\zeta  \, | \, \tau)}{\theta_{\frac un,\frac vn}(\zeta \, | \, \tau)} \, , 
\end{equation}
where
\begin{equation}\label{eq:4}
  \zeta  \; :=\;  \eta + \tfrac{1}{2}(\tau+1).
\end{equation}
We denoted $\eta + \tfrac{1}{2}(\tau+1)$ by $\xi$ in \cite{CKS4},  but in this paper $\xi$ always denotes an arbitrary point on $E$.

The next result follows immediately from \Cref{prop.theta.uv}.

\begin{proposition}
\label{prop.wuv.properties}
The functions $w_{(u,v)}(z)$ have the following properties:
\begin{enumerate}
  \item
 $  w_{(u,v)}(0)=1$;  
  \item\label{eq.w.qper}
if $s,t \in \ZZ$, \ then
\begin{equation*}
\phantom{xxxx} \frac{ w_{(u,v)}(z+s\tau +t)}{w_{(u,v)}(z)}  \; = \; e\left(- s(z+\eta) -\tfrac{s}{2} (s\tau + \tau+1)   + \tfrac{1}{n}(tu -sv) \right);
\end{equation*}
\item\label{eq.w.mod}
$w_{(u,v)}(z)=0$ if and only if $z=-\eta -\frac{1}{n}(u\tau+v)$ modulo $\Lambda_\tau$;
  \item\label{eq.w.simple}
  each of these zeros is a simple zero.
\end{enumerate}
\end{proposition}

In particular, $w_{(u,v)} (z+1) = e(\frac{u}{n}) w_{(u,v)} (z)$ and $w_{(u,v)} (z+\tau) = e(-z-\eta-\tau -\frac{1}{2} - \frac{1}{n}v) w_{(u,v)} (z)$.

\subsubsection{}
The notation $\theta_{u,v}$ and $w_{(u,v)}$ makes sense for arbitrary complex numbers $u$ and $v$. 
As explained in the discussion preceding Theorem 3.1 in \cite{CKS4}, when $u$ and $v$ are integers $w_{(u,v)}(z)$ depends only on their images in $\ZZ_n$ so, in such a case, we will think of the subscript in $w_{(u,v)}(z)$ as an element of $\bZ_n^2$.

\subsection{The action of $\SL(2,\bbZ)$ on $\CC \times \CC \times \HH$}

If $M={{a \; \,b} \choose {c \, \; d}} \in \GL(2,\RR)$ and $\tau \in \HH$ we define
\begin{equation*}
M \triangleright \tau \; := \; \frac{a\tau+b}{c\tau+d} \, .
\end{equation*}
If $\Im(z)$ denotes the imaginary part of a complex number $z$,
$ \Im(M \triangleright \tau)= \det(M) \, |c\tau+d|^{-2} \Im(\tau)$. 
Hence the group $\GL^+(2,\bR)$ of $2 \times 2$ real matrices with positive determinant acts on $\bH$. 
That action extends to actions on $\bC\times\bH$, given by  the formula
\begin{equation}\label{eq:gl2act}
  \begin{pmatrix}
    a&b\\
    c&d
  \end{pmatrix}
  \,
  \triangleright
  \, 
  (z \, | \, \tau)
  \; := \;
  \left(\frac{z}{c\tau+d} \; \bigg \vert \; \frac{a\tau+b}{c\tau+d}\right),
\end{equation}
and on $\bC \times \bC\times\bH$, given by  the formula
\begin{equation}\label{eq:actredux}
  \begin{pmatrix}
    a&b\\
    c&d
  \end{pmatrix}
\,  \triangleright \, 
(z,\eta \, | \, \tau)
  \; := \;
  \left(\frac{z}{c\tau+d}\,, \, \frac{\eta}{c\tau+d} \; \bigg| \; \frac{a\tau+b}{c\tau+d}\right).
\end{equation}
In particular, the modular group $\SL(2,\bZ)$ acts on $\bC\times\bH$ and $\bC \times \bC\times\bH$. 
The  action on $\bC\times\bH$  is the same as the action in \cite[(2.12)]{ric-tra}, 
which is a source we often refer to,  both here and in \cite{CKS4}.

\subsection{Equivariance properties of $w_{(u,v)}(z,\eta \, | \, \tau)$}
 To state the next two results, we make explicit the dependence of $w_{(u,v)}$ on all parameters by writing
\begin{equation}\label{eq:wtrip}
  w_{(u,v)}(z,\eta \, | \, \tau) \; := \;  w_{(u,v)}(z).
\end{equation}

\begin{proposition}\label{pr:wev}
  Let  $M ={{a \; \,b} \choose {c \, \; d}} \in \SL(2,\bZ)$ act on the triples $(z,\eta \, | \, \tau)$ via \Cref{eq:actredux}.
  The functions
  \begin{equation}\label{eq:onab}
    z \, \mapsto \, w_{(u,v)M}(z,\eta \, | \, \tau)
  \end{equation}
 and
  \begin{equation}\label{eq:onzt}
    z\, \mapsto \, w_{(u,v)}\big(M \triangleright (z,\eta \, | \, \tau)\big)
  \end{equation}
  have the same zeros with the same multiplicities.
\end{proposition}
\begin{proof}
Let
\begin{equation*}
(z',\eta' \, | \, \tau') \; :=\; M \triangleright (z,\eta \, | \, \tau) \;=\;  \left(\frac{z}{c\tau+d},\ \frac{\eta}{c\tau+d}\ \bigg| \ \frac{a\tau+b}{c\tau+d}\right).
\end{equation*}
By \Cref{prop.wuv.properties}\cref{eq.w.mod}, $w_{(u,v)M}(z,\eta \, | \, \tau)=0$ if and only if 
  \begin{equation*}
   z+\eta + \tfrac{1}{n}\big( (au+cv)\tau+(bu+dv)\big)  \, \in \,  \Lambda_{\tau}  \;  = \; \ZZ(a\tau+b) + \ZZ(c\tau+d),    
  \end{equation*}
   where the  equality uses the fact that $M \in \SL(2,\bZ)$. 
 After dividing by $c\tau+d$, this condition becomes
    \begin{equation*}
   z'+\eta' + \tfrac{1}{c\tau+d} \,  \tfrac{1}{n} \,  \big( u(a\tau+b) +v(c\tau+d) \big)  \, \in \,   \ZZ\tau' +\ZZ \; =\; \Lambda_{\tau'}.
  \end{equation*}
Therefore $w_{(u,v)M}(z,\eta \, | \, \tau)=0$ if and only if   
     \begin{equation*}
   z'+\eta' +  \tfrac{1}{n} \big( u \tau' +v  \big)  \, \in \, \Lambda_{\tau'}.
  \end{equation*} 
  By \Cref{prop.wuv.properties}\cref{eq.w.mod}, this happens if and only if $w_{(u,v)}(z', \eta' \, | \, \tau')=0$; i.e., 
  if and only if $w_{(u,v)}\big(M \triangleright (z,\eta \, | \, \tau)\big) =0$. Thus the zeros of the functions are the same. By \Cref{prop.wuv.properties}\cref{eq.w.simple} all the zeros are simple zeros so their multiplicities are the same.
\end{proof}

We will now improve \Cref{pr:wev}. First, we recall some standard terminology.

\subsubsection{Quasi-periodicity}
Suppose that $\{\omega_1,\omega_2\}$ is an $\RR$-basis for $\CC$, and let $\L:=\ZZ\omega_1+\ZZ\omega_2$.
We say that a function $f:\CC \to \CC$ is {\sf quasi-periodic with respect to $\L$} if there are elements $a,b,c,d\in \CC$ such that
\begin{align*}
f(z+\omega_1) & \;=\; e^{2\pi i (az+b)}f(z) \qquad \hbox{and} \qquad
\\
f(z+\omega_2) & \;=\; e^{2\pi i (cz+d)}f(z)
\end{align*}
for all $z \in \CC$.

\subsubsection{Factors of automorphy}
A  {\sf factor of automorphy} for $\Lambda_\tau$ is a function
\begin{equation*}
  \Lambda_{\tau }\times \bC \, \longrightarrow \, \CC^\times, \qquad  (\l,z) \mapsto e_{\l}(z),
\end{equation*}
such that $z\mapsto e_{\lambda}(z)$ is a nowhere-vanishing holomorphic function on $\bbC$ and
\begin{equation*}
e_{\l+\mu}(z) \; = \; e_\l(z+\mu)e_\mu(z)
\end{equation*}
for all $\l, \mu \in \L_\tau$; i.e.,  the function $\lambda \mapsto e_\l(-)$ is a 1-cocycle on 
$\Lambda_\tau$ taking values in the group of nowhere-vanishing holomorphic functions on $\CC$, where
 the latter is a representation of $\L_\tau$ via the action $(\l \cdot f)(z):=f(z+\l)$. 

If $f(z \, | \, \tau)$ is quasi-periodic in $z$ with respect to $\ZZ+\ZZ\tau$, then its quasi-periodicity properties  imply that the function
\begin{equation*}
(s\tau+t,z) \, \mapsto \, e_{s\tau+t}(z) \; :=\; \frac{f(z+s\tau+t \, | \, \tau)} {f(z \, | \, \tau)} \qquad  \qquad (s,t\in \ZZ)
\end{equation*}
is a factor of automorphy for $\ZZ+\ZZ\tau$. We call it {\sf the factor of automorphy for $f(z)$}. 
For example, \Cref{prop.wuv.properties}\cref{eq.w.qper} says that the factor of automorphy for $w_{(u,v)}(z,\eta \, | \, \tau)$ is
\begin{equation*}
  e\left( - s(z+\eta) -\tfrac{s}{2} (s\tau + \tau+1)  + \tfrac{1}{n}(u,v) \begin{pmatrix} t \\ -s \end{pmatrix} \right).
\end{equation*}

\subsubsection{}
\label{ssect.gen.princ}
The following fact will be used in the next proof:
 if $f(z \, | \, \tau)$ and $g(z \, | \, \tau)$ are quasi-periodic with respect to $\ZZ+\ZZ\tau$ 
having the same factors of automorphy and the same zeros (counted with multiplicity), then there is a constant $c$ such that
$f(z \, | \, \tau)=c\, g(z \, | \, \tau)$ for all $z \in \CC$. The reason is that $f(z \, | \, \tau)/g(z \, | \, \tau)$ is an elliptic function with respect to
$\Lambda_\tau$ having no zeros and no poles, and hence constant. 
 
\subsubsection{Vague remark}
The terminology ``factor of automorphy'' is also used in the context of line bundles on the elliptic curve $E_{\tau}:=\bC/\Lambda_{\tau}$. Theta functions with respect to $\Lambda_{\tau}$ are essentially the same things as sections of line bundles on $E_\tau$
 (see, e.g., \cite[p.~24 and Appendix~B]{bl}).

\begin{theorem}\label{th:wuvs}
Let $M\in \SL(2,\bZ)$. 
There is a nowhere-vanishing holomorphic function $f(z)=f_M(z)$, that does not depend on $(u,v)$, such that
  \begin{equation*}
    w_{(u,v)M}(z,\eta \, |  \, \tau) \; = \;  f(z) \,w_{(u,v)}\big(M\triangleright (z,\eta \, |  \, \tau)\big).
  \end{equation*}
\end{theorem}
\begin{proof}
  If the theorem holds for $M$ and $N$, then it holds for $MN$ with $f_{MN}(z)=f_M(z)f_N(z)$ and for $M^{-1}$ with $f_{M^{-1}}(z)=f_{M}(z)^{-1}$. 
 It therefore suffices to prove the theorem for each of the two generators (\cite[\S 1.5.3]{trees})
  \begin{equation*}
    X:=\begin{pmatrix}
     0&-1\\
      1& \phantom{-}0
    \end{pmatrix}
  \qquad  \text{ and } \qquad
    Y:=\begin{pmatrix}
      1&1\\
      0&1
    \end{pmatrix}
  \end{equation*}
  of $\SL(2,\bZ)$.  
To prove the theorem it suffices to show that $w_{(u,v)M}(z,\eta \, |  \, \tau)$
and $w_{(u,v)}\big(M\triangleright (z,\eta \, |  \, \tau)\big)$ have the same quasi-periodicity properties with respect to $\Lambda_{\tau}$,
and the  same zeros. We already showed in \Cref{pr:wev} that 
they have the same zeros, so it remains to show that they have the same factors of automorphy with respect to $\Lambda_\tau$.

Let  $M={{a \; \,b} \choose {c \, \; d}} \in \SL(2,\ZZ)$. 
As in the previous proof, we adopt the notation 
\begin{align*}
  (z',\eta' \, | \, \tau')  & \; := \; M \triangleright(z,\eta \, | \, \tau) \;=\;  \left(\frac{z}{c\tau+d},\ \frac{\eta}{c\tau+d}\ \bigg| \ \frac{a\tau+b}{c\tau+d} \right).
\end{align*}

{\sf Claim.} There are unique integers $s',t'$ such that $(z'+s'\tau'+t',\eta' \, | \, \tau') = M\triangleright(z+s\tau +t,\eta \, | \,  \tau)$,
  namely $s':=ds-ct$ and $t':=at-bs$, i.e., $ { \;\; t' \choose \, -s' \, } = M  { \;\; t \choose \, -s \, }$. 

 {\sf Proof.}
 By definition, 
 \begin{equation*}
 M\triangleright(z+s\tau +t,\eta \, | \,  \tau)   \;=\;  \left(\frac{z+s\tau+t}{c\tau+d},\ \frac{\eta}{c\tau+d}\ \bigg| \ \frac{a\tau+b}{c\tau+d} \right).
\end{equation*}
Since $z'=\frac{z}{c\tau+d}$ it suffices to show there are unique integers $s',t'$ such that 
$s'\tau'+t' \;=\; \tfrac{s\tau+t}{c\tau+d}$. Certainly, if such an $(s',t')\in \ZZ^2$ exists it is unique because $\{1,\tau'\}$ is linearly
independent over $\ZZ$---even over $\RR$.  
Since $\tau'=\frac{a\tau+b}{c\tau+d}$, a simple calculation shows that  $(ds-ct)\tau'+(at-bs) \;=\; \tfrac{s\tau+t}{c\tau+d}$: just
multiply both sides by $c\tau+d$ and use the fact that $ad-bc=1$.\footnote{Alternatively, if $(s',t') \in \RR^2$ is such that $(c\tau+d)(s',t') {{\tau'} \choose {1}} = (s,t) {{\tau} \choose {1}}$,
then
$(s,t)  {{\tau} \choose {1}}=(s',t')  {{a\tau+b} \choose {c\tau+d}} =  (s',t') \,  M {{\tau} \choose {1}}$
and, since $\{1,\tau\}$ is linearly independent over $\RR$,  
$(s',t') =(s,t)M^{-1}= (s,t) {{\;\; d \; \,-b} \choose {-c \, \;\, \; a}}=(ds-ct,at-bs)$. 
} 
$\lozenge$

By \Cref{prop.wuv.properties}\cref{eq.w.qper}, the factor of automorphy for $w_{(u,v)M}(z\, | \, \tau)$
is
\begin{equation}\label{eq:1fact}
 e\left( - s(z+\eta)  -\tfrac{s}{2} (s\tau + \tau+1) + \tfrac{1}{n} \, (u,v) \, M \begin{pmatrix} t \\ -s \end{pmatrix} 
  \right).
\end{equation}
By \Cref{prop.wuv.properties}\cref{eq.w.qper}, the factor of automorphy for $w_{(u,v)}\big(M\triangleright (z,\eta \, |  \, \tau)\big)$,
 i.e., the ratio
\begin{equation*}
  \frac
  {w_{(u,v)}\left(M\triangleright (z+s\tau+t,\eta \, | \, \tau)\right)}
  {w_{(u,v)}\left(M \triangleright (z,\eta \, | \, \tau)\right)}
\;  = \;
  \frac{      w_{(u,v)}(z'+s'\tau'+t',\eta' \, | \, \tau')    }
  {     w_{(u,v)}  (z' , \eta' \, | \, \tau')   },
\end{equation*}
is 
\begin{equation}\label{eq:1factother}
e\left(- s'(z'+\eta') -\tfrac{s'}{2} (s'\tau' + \tau'+1)   + \tfrac{1}{n}(u,v) \begin{pmatrix} t' \\ -s'\end{pmatrix} \right).
\end{equation}
But the terms involving $(u,v)$ in \Cref{eq:1fact} and \Cref{eq:1factother}  are equal,  
so
\begin{equation}
\label{eq:ratio}
\frac{ \text{\Cref{eq:1fact}}} { \text{\Cref{eq:1factother}} }   \;=\;
\frac{  e\left( - s(z+\eta)  -\tfrac{s}{2} (s\tau + \tau+1)  \right) } {  e\left(- s'(z'+\eta')-\tfrac{s'}{2} (s'\tau' + \tau'+1)  \right) } \,.
\end{equation}
  
  (1) 
 (Proof   for $M=Y$.)
 In this case,   $(z',  \eta' \, | \, \tau')= (z,\eta \, | \, \tau+1)$ and  $(t',-s')  = (t,-s)Y^{t} = (t-s,-s)$.
  Thus, when $M=Y$, the  denominator in \Cref{eq:ratio}  is
\begin{equation*}
 e\left( - s(z+\eta) -\tfrac{s}{2} (s\tau +s + \tau +1+1) \right),
\end{equation*}
 which is  equal to the numerator  in \Cref{eq:ratio} because $e\big( -\tfrac{s}{2} (s+1) \big)=1$.
Hence \Cref{eq:1fact}=\Cref{eq:1factother} when $M=Y$ so,  by  the remark in \Cref{ssect.gen.princ}, 
 $w_{(u,u+v)}(z,\eta \, | \, \tau) = c\, w_{(u,v)}(z,\eta \, | \, \tau+1)$ for some constant $c$. But $w_{(u,u+v)}(0,\eta \, | \, \tau) = 1
 = w_{(u,v)}(0,\eta \, | \, \tau+1)$ so $c=1$, and we conclude that the theorem holds for $Y$ with $f_{Y}(z)=1$ for all $z$. 

 (2)
 (Proof for $M=X$.)
   In this case,  $(u,v)M=(v,-u)$, $ (z',\eta' \, | \, \tau') = (z/\tau, \,  \eta/\tau \, | \, -1/\tau)$,    and $ (s',t')  = (-t,s)$.
 Thus, when $M=X$,
\begin{align}
\label{eq:ratio.fa}
\frac{ \text{\Cref{eq:1fact}}} { \text{\Cref{eq:1factother}} }  & \;=\;
e\left( - s(z+\eta) -\tfrac{s}{2} (s\tau + \tau+1)   - \tfrac{t}{2} (\tfrac{t}{\tau}-\tfrac{1}{\tau} +1) - \tfrac{t}{\tau}(z+\eta)\right). 
\end{align}
If we replaced the term $w_{(u,v)M}(z,\eta \, |  \, \tau)$  by $g(z)\, w_{(u,v)M}(z,\eta \, |  \, \tau)$
where $g(z)=e(Az^2+Bz)$, then \Cref{eq:ratio.fa} would be multiplied by
\begin{equation*}
\frac{g(z+s\tau+t)}{g(z)} \;=\; \frac{e(A(z+s\tau+t)^2+B(z+s\tau+t))}{e(Az^2+Bz)} \;=\; 
e\big(A(2z+s\tau+t) (s\tau+t) + B(s\tau+t)\big).
\end{equation*}
Hence if $A=\frac{1}{2\tau}$ and $B=\frac{1}{2}-\frac{1}{2\tau}+\frac{\eta}{\tau}$, 
then \Cref{eq:ratio.fa}  would be multiplied by
\begin{equation*}
e\big(\tfrac{1}{2\tau}(2z+s\tau+t) (s\tau+t) + \big(\tfrac{1}{2}-\tfrac{1}{2\tau}+\tfrac{\eta}{\tau}\big)(s\tau+t)\big)
\end{equation*}
which equals 
\begin{equation*}
e\big( 
sz +\tfrac{zt}{\tau} +\tfrac{1}{2\tau}( s^2\tau^2 + 2 st\tau+t^2)  + \tfrac{s\tau}{2}- \tfrac{s}{2} +s\eta 
+ \tfrac{t}{2}-\tfrac{t}{2\tau}+\tfrac{t\eta}{\tau}
\big).
\end{equation*}
A straightforward calculation shows that the product of this with  \Cref{eq:ratio.fa}  equals 1, so we conclude that
 $g(z)\, w_{(u,v)X}(z,\eta \, |  \, \tau)$ and $w_{(u,v)}\big(X\triangleright (z,\eta \, |  \, \tau)\big)$ have the same factors of
 automorphy with respect to $\L_\tau$. Since $g(z) \ne 0$ for all $z$, they also have the same zeros by  \Cref{pr:wev}, and both take the 
 value 1 when $z=0$, so we conclude that 
 \begin{equation*}
  w_{(u,v)X}(z,\eta \, |  \, \tau) \;=\; 
e\big( - \tfrac{1}{2\tau}z^2 + \big(\tfrac{1}{2\tau} - \tfrac{1}{2} - \tfrac{\eta}{\tau}\big)   z \big) \,   w_{(u,v)}\big(X\triangleright (z,\eta \, |  \, \tau)\big).
\end{equation*}

 The proof is complete.
\end{proof}

\section{An extension of the finite Heisenberg group}
\label{sect.Hn}

Heisenberg groups play an important  organizational role in the theory of theta functions and abelian varieties---see, for example, 
\cite[Ch.~6]{bl}, \cite[Ch.~I, \S3]{Mum07},  \cite{mum-tata-III}, and \cite{Polishchuk-book}. 
In the present setting, the relevant Heisenberg group is an extension $1 \to \mu_n \to H_n \to \ZZ_n^2 \to 0$ where
$\mu_n$ is the group of complex $n^{\th}$ roots of unity. Below we will make $V$, the degree-one component of $Q_{n,k}(\eta \, | \, \tau)$,
an irreducible representation for $H_n$. As Feigin and Odesskii first noticed, that action lifts to an action of $H_n$ as automorphisms of 
$Q_{n,k}(\eta \, | \, \tau)$ (see \cite[\S\S2.3 and 3.5]{CKS1} for details).

\Cref{prop.GL2.non-action} shows that when $n$ is odd the natural action of each 
$M \in \SL(2,\ZZ)$ on $\ZZ_n^2=H_n/\langle \epsilon \rangle$ 
lifts to an action of $M$ as an automorphism of $H_n$.
(But this does not yield a homomorphism $\SL(2,\ZZ) \to \Aut(H_n)$.) Such a lifting does not exist when $n$ is even. 
For this reason we introduce an index-two extension $\widetilde{H}_n$ of $H_n$  for which such a lifting does exist. The utility of 
$H_n$ and $\widetilde{H}_n$  becomes apparent in \cref{se:main} when we introduce a new presentation for 
 the linear operator  $R_{n,k}(z,\eta \, | \, \tau):V^{\otimes 2} \to V^{\otimes 2}$ in terms of a new linear operator 
 $T_k(z,\eta \, | \, \tau):V^{\otimes 2} \to V^{\otimes 2}$ that
transforms nicely with respect to the actions of, first, $\SL(2,\ZZ)$ on $(z,\eta \, | \, \tau)$ and the subscripts $(u,v) \in \ZZ_n^2$
 in $w_{(u,v)}(z)$,    and, second,  $\widetilde{H}_n$ on $V^{\otimes 2}$, and, third,   an action of $\SL(2,\ZZ)$ on  $\widetilde{H}_n$.

\subsection{The Heisenberg group $H_n$}
The Heisenberg group of order $n^3$ is 
\begin{equation*}
  H_n \; := \; 
  \left\langle
    S, T, \epsilon\ |\ S^n=T^n=\epsilon^n=1,\ [S,\epsilon]=[T,\epsilon]=1,\ [S,T]=\epsilon
  \right\rangle.
\end{equation*}

The following equalities are useful when computing in $H_n$:
\begin{align*}
S^aT^b & \; =\; T^bS^a \epsilon^{ab},
\\
S^{a_1}T^{b_1} S^{a_2}T^{b_2} S^{a_3}T^{b_3} \cdots & \; =\; T^B S^A \epsilon^{C},
\quad \text{where}  \; \; A=a_1+a_2+\cdots,  \quad B=b_1+b_2+\cdots,
\\
  & \phantom{xxxxxxxxxxxxxx} C=a_1b_1+(a_1+a_2)b_2+(a_1+a_2+a_3)b_3 + \cdots,
\\  
(S^aT^b)^m & \; =\; T^{bm}S^{am} \epsilon^{X},  \quad \text{where}  \; \; X= \tfrac{1}{2}m(m+1)ab,
\\  
(T^bS^a)^m & \; =\; T^{bm}S^{am} \epsilon^{Y},  \quad \text{where}  \; \; Y= \tfrac{1}{2}m(m-1)ab.
\end{align*}

\begin{proposition}
\label{prop.GL2.non-action}
Let $M={{a \; \,b} \choose {c \, \; d}} \in \GL(2,\ZZ)$. If $n$ is odd, there is an automorphism 
$\Psi_M':H_n \to H_n$ such that
\begin{equation}
\label{eq:bad.map}
  \Psi_M': \qquad    T \mapsto   T^aS^c,  \qquad   S \mapsto  T^bS^d, \qquad  \epsilon
  \mapsto \epsilon^{\det M}.  
\end{equation}
\end{proposition}
\begin{proof}
  The expressions $T^aS^c$ and $T^bS^d$ make sense when $a,b,c,d \in \ZZ_n$ because $T$ and $S$ have order $n$.  The result follows from the fact that $T^aS^c$ and $T^bS^d$ have order $n$, and
\begin{equation*}
\big(T^bS^d \big) \, \big(T^aS^c \big) \big(T^bS^d \big)^{-1} \, \big(T^aS^c \big)^{-1} 
\;=\; \epsilon^{ad-bc}.  
\end{equation*}
The calculations contain no surprises.
\end{proof}

The map $\GL(2,\ZZ) \to \Aut(H_n)$, $M \mapsto \Psi'_M$,  is {\it not} a group homomorphism. 
For example, if $M={{0 \; \,-1} \choose {1\; \; \, \; 1}}$, then  
 $\Psi_M(T)=S$, $\Psi_M(S)=T^{-1}S$, $\Psi_M \Psi_M(S)= S^{-1}T^{-1}S= \epsilon T^{-1}$, but $\Psi_{M^2}(S)=T^{-1}$.  Similarly, if $M={{0 \; \,1} \choose {1 \, \; 0}}$ and $N={{1 \; \,0} \choose {1 \, \; 1}}$, then $\Psi_M\Psi_N(T)=\Psi_M(TS)=ST \ne TS=\Psi_{MN}(T)$.

Because  $M \mapsto \Psi'_M$ is not a group homomorphism we will make no use of $\Psi'_M$.

When $n$ is even, an additional problem arises: the map $\Psi_M'$ in \Cref{eq:bad.map} does not extend to  an
automorphism of $H_n$ because $T^aS^c$ and $T^bS^d$ need not have order $n$. 
For example, $T S$ has order $2n$.\footnote{ If $b$ and $c$ are even residues modulo $n$, however, both $T^a S^c$ and $T^b S^d$ have order $n$.}  In order to address this problem we introduce a new group when $n$ is even, 
the group $    \widetilde{H}_n$ defined below.

\subsection{The extension $\widetilde{H}_n$ of $H_n$ }\label{subse:act}

When $n$ is even we define
\begin{equation*}
    \widetilde{H}_n \; := \; 
    \left\langle
    S, T, \epsilon^{1/2}\ \; \Big\vert \; S^n=T^n=\epsilon^n=1,\ [S,\epsilon^{1/2}]=[T,\epsilon^{1/2}]=1,\ [S,T]=\epsilon=(\epsilon^{1/2})^2
  \right\rangle.
  \end{equation*} 
  Here $\epsilon^{1/2}$ is a symbol that behaves like a square root of $\epsilon$.  We note that $\langle S, T, \epsilon\rangle$ is a normal subgroup of index two in $\widetilde{H}_n$ and is isomorphic to $H_n$.

 \subsubsection{The notation $\widetilde{H}_n$ when $n$ is odd}
 \label{ssect.H-tilde.notn}
  For convenience, we define $\widetilde{H}_n:=H_n$ when $n$ is odd, with the convention that $\epsilon^{1/2}:=\epsilon^{(n+1)/2}$. 
  In that case,  $(\epsilon^{(n+1)/2})^{2}=\epsilon$, so $\langle\epsilon^{(n+1)/2}\rangle=\langle\epsilon\rangle$, and $(\epsilon^{1/2})^{-1}=\epsilon^{(n-1)/2}$.

\subsubsection{The action of $\widetilde{H}_n$ as automorphisms of $Q_{n,k}(\eta \, | \, \tau)$}
Feigin and Odesskii observed  that $H_n$ acts as automorphisms of  $Q_{n,k}(\eta \, | \, \tau)$
\cite[p.~1143]{Od-survey} via the actions
\begin{equation*}
S \cdot x_i = \omega^i x_i, \qquad  T \cdot x_i = x_{i+1}, \qquad  \epsilon \cdot x_i =  e\big(\tfrac{1}{n}\big) x_i.
\end{equation*}
A proof of this can be found at \cite[Prop.~3.23]{CKS1}. 
It is easy to see that this extends to an action of $\widetilde{H}_n$ as automorphisms of 
$Q_{n,k}(\eta \, | \, \tau)$ when $n$ is even.

\begin{proposition}
\label{prop.qnk.act}
The group $\widetilde{H}_n$ acts as degree-preserving $\CC$-algebra automorphisms of $Q_{n,k}(\eta \, | \, \tau)$ by 
\begin{equation}
\label{Hberg.action.on.V}
S \cdot x_i = \omega^i x_i, \qquad  T \cdot x_i = x_{i+1}, \qquad  \epsilon^{1/2} \cdot x_i = - e\big(\tfrac{1}{2n}\big) x_i
=  e\big(\tfrac{n+1}{2n}\big) x_i.
\end{equation}
\end{proposition}
 
 \subsubsection{}
The reason for choosing $\epsilon^{1/2} \cdot x_i = - e\big(\tfrac{1}{2n}\big) x_i$ rather than $\epsilon^{1/2} \cdot x_i =  e\big(\tfrac{1}{2n}\big) x_i$ in \cref{Hberg.action.on.V} is so the formula works for both even and odd $n$; i.e., 
the choice we have made for $\epsilon^{1/2} \cdot x_i $ is
compatible with the convention in \Cref{ssect.H-tilde.notn} that $\epsilon^{1/2}=\epsilon^{(n+1)/2}$ when $n$ is odd,
but the other choice is incompatible.

 \subsection{The action of $\SL(2,\bZ)$ as automorphisms of $\widetilde{H}_n$}

There are many actions of $\SL(2,\bZ)$ as automorphisms of $\widetilde{H}_n$,
but we single out one particular action in \Cref{prop.aut.Hn.tilde} that has the virtue of simplicity.

\begin{lemma}
\label{lem.automs.PsiM}
Let $\nu =\epsilon^{1/2}$ when $n$ is even and let $\nu=\epsilon^{(n+1)/2}$ when $n$ is odd.
If $M={{a \; \,b} \choose {c \, \; d}} \in \GL(2,\ZZ)$, then there is an automorphism  $\Psi_M$ of $\widetilde{H}_n$
such that 
\begin{align*}
\Psi_M  \, : \qquad &   T\mapsto T^a S^c \nu^{ac},  \qquad S\mapsto T^b S^d \nu^{bd},  \qquad \nu \mapsto \nu^{ad-bc},
\\
& T^mS^r\mapsto T^{am+br}S^{cm+dr}\nu^Z ,
\end{align*}
where $Z=acm^2+bdr^2+2bcmr$.
\end{lemma}
\begin{proof}
Since
\begin{align*}
\Psi_M(S)\Psi_M(T) & \;=\; T^b S^d \nu^{bd}  T^a S^c \nu^{ac} \;=\; T^{a+ b} S^{c+d}  \epsilon^{ad} \nu^{bd+ac} 
\qquad \text{and}
\\
\Psi_M(T)\Psi_M(S) & \;=\;  T^a S^c \nu^{ac} T^b S^d \nu^{bd}  \;=\; T^{a+ b} S^{c+d}  \epsilon^{bc} \nu^{bd+ac},
\end{align*}
we have $\Psi_M(S)\Psi_M(T)=\Psi_M(T)\Psi_M(S) \epsilon^{ad-bc} =\Psi_M(T)\Psi_M(S) \Psi_M(\epsilon)  $.

Furthermore,
\begin{equation*}
\Psi_M(T)^n  \;=\; ( T^a S^c \nu^{ac})^n \;=\;   T^{an} S^{cn} \epsilon^D  \nu^{acn}
\end{equation*}
where $D= \tfrac{1}{2} n(n-1)ac$. If $n$ is odd, then $D$ is an integer multiple of $n$ so $\epsilon^D=1$,
and $ \nu^{acn} = \epsilon^{acn(n+1)/2} =1$ so $\Psi_M(T)^n  =1$.
If $n$ is even, then $D= m(n-1)ac$ where $m=n/2$ and $\nu^{acn}=\epsilon^{acm}$, whence 
$ \epsilon^D  \nu^{acn} = \epsilon^{mnac}=1$; therefore $\Psi_M(T)^n =1$ when $n$ is even.
In conclusion, $\Psi_M(T)^n =1$ when $n$ is odd and when $n$ is even.
Similar calculations show that $\Psi_M(S)^n =1$  when $n$ is odd and when $n$ is even.

Hence $\Psi_M$ is a group homomorphism.  

Clearly, $\ker(\Psi_M) \cap \langle \nu \rangle = \{1\}$. However, every non-trivial normal subgroup of 
$\widetilde{H}_n$ has non-trivial intersection with $\langle \nu \rangle $ so we conclude that $\ker(\Psi_M) = \{1\}$. 
Hence $\Psi_M$ is an automorphism of $\widetilde{H}_n$. 

We have
\begin{align*}
\Psi_M(T^mS^r) & \;=\; (T^aS^c\nu^{ac})^m(T^bS^d\nu^{bd})^r 
\\
& \;=\; T^{am}S^{cm}\epsilon^X  T^{br}S^{dr}\epsilon^Y \nu^{acm+bdr}
\\
& \;=\; T^{am+br}S^{cm+dr}\epsilon^{X+Y+bcmr} \nu^{acm+bdr}
\end{align*}
where $X=\frac{1}{2}m(m-1)ac$ and $Y=\frac{1}{2}r(r-1)bd$. Since $\epsilon=\nu^{2}$, the last part of the lemma holds with
$Z=m(m-1)ac+r(r-1)bd+2bcmr+acm+bdr$. Hence the result.
\end{proof}

\begin{proposition}
\label{prop.aut.Hn.tilde}
Let $\nu =\epsilon^{1/2}$ when $n$ is even and let $\nu=\epsilon^{(n+1)/2}$ when $n$ is odd.
The map $\Psi:\SL(2,\ZZ) \to \Aut(\widetilde{H}_n)$ that sends $M={{a \; \,b} \choose {c \, \; d}} \in \SL(2,\ZZ)$ to
the automorphism
\begin{equation*}
\Psi_M  \, : \qquad T\mapsto T^a S^c \nu^{ac},  \qquad S\mapsto T^b S^d \nu^{bd},  \qquad \nu \mapsto \nu,
\end{equation*}
is a group homomorphism.
\end{proposition}
\begin{proof}
Let $N={{a' \; \,b'} \choose {c' \, \; d'}} \in \SL(2,\ZZ)$. 
Then
\begin{align*}
(\Psi_N \Psi_M)(T) & \; = \; (T^{a'} S^{c'} \nu^{a'c'})^a\, (T^{b'} S^{d'} \nu^{b'd'})^c \nu^{ac}
\\
& \; = \; T^{a'a} S^{c'a} \epsilon^A \nu^{a'c'a}\, T^{b'c} S^{d'c}  \epsilon^B \nu^{b'd'c} \nu^{ac}
\\
& \; = \; T^{a'a+b'c} S^{c'a+d'c} \epsilon^{c'ab'c} \epsilon^{A+B} \nu^{a'c'a+b'd'c+ac} \qquad \text{and}
\\
\Psi_{NM}(T) & \; = \;  T^{a'a+b'c} S^{c'a+d'c} \nu^{(aa'+b'c)(c'a+d'c)}
\end{align*}
where $A=\frac{1}{2}a(a-1)a'c'$ and $B=\frac{1}{2}c(c-1)b'd'$.
We have  
\begin{equation*}
\nu^{(aa'+b'c)(c'a+d'c)}\nu^{-a'c'a-b'd'c-ac} \; = \; \nu^D
\end{equation*}
where 
\begin{align*}
D & \; = \; a'c'a^2+a'd'ac +b'c'ac+b'd' c^2 -a'c'a-b'd'c-ac
\\
& \;=\; ac(a'd'+b'c'-1) +a'c'(a^2-a) + b'd'(c^2-c)
\\
& \;=\; 2acb'c' +a(a-1)a'c'+ c(c-1)b'd' 
\\
& \;=\; 2(acb'c' +A+B).
\end{align*}
Since $\nu^2=\epsilon$, $\nu^D=\epsilon^{acb'c'+A+B}$. Hence $(\Psi_N \Psi_M)(T) =\Psi_{NM}(T)$.
A similar calculation shows that  $(\Psi_N \Psi_M)(S) =\Psi_{NM}(S)$.
Hence $\Psi_N \Psi_M =\Psi_{NM}$, so $\Psi$ is a group homomorphism.
\end{proof}

\subsubsection{Remark}
There is another way to obtain the homomorphism $\Psi:\SL(2,\ZZ) \to \Aut(\widetilde{H}_n)$. 

By  \cite[\S 1.5.3]{trees} or \cite[Prop.~2.1]{ccs}, for example, $\SL(2,\bZ)$ is the amalgamated free product 
$ \bZ_4*_{\bZ_2}\bZ_6$  with $\bZ_4$ and $\bZ_6$ generated by 
\begin{equation}\label{eq:gens}
  X \; := \; \begin{pmatrix}
    \phantom{-}0 & 1 \\
    -1 & 0
  \end{pmatrix}
\qquad   \text{ and } \qquad
 Y \; := \;\begin{pmatrix}
    \phantom{-}1 & 1 \\
    -1 & 0
  \end{pmatrix},
\end{equation}
respectively.
Independently of \Cref{lem.automs.PsiM}, it is easy to check that the formulas
\begin{align*}
\Psi_X  \, &  : \qquad T\mapsto S^{-1},  \qquad S\mapsto T,  \qquad \nu \mapsto \nu,
\\
\Psi_Y \, &  : \qquad T\mapsto T S^{-1}\nu^{-1},  \qquad S\mapsto T,  \qquad \nu \mapsto \nu,
\end{align*}
(where $\nu=\epsilon^{1/2}$) extend to automorphisms of $\widetilde{H}_n$.
Since  
\begin{equation*}
\Psi_X^2 \, = \, \Psi_Y^3 \,  : \qquad T\mapsto T^{-1},  \qquad S\mapsto S^{-1},  \qquad \nu \mapsto \nu,
\end{equation*}
$\Psi_X^4  =  \Psi_Y^6 = \id$. Hence there is a unique group homomorphism 
$\Psi:\bZ_4*_{\bZ_2}\bZ_6 = \langle X \rangle*_{\bZ_2} \langle Y \rangle \to \Aut{H}_n$
with the property that $\Psi(X)=\Psi_X$ and $\Psi(Y)=\Psi_Y$. 

The drawback to this alternative approach is that some additional calculation is needed to produce a formula for 
$\Psi(M)$ for all $M \in \SL(2,\ZZ)$.

\subsubsection{Remark}
\label{ssect.Z.vs.Zn}
In \Cref{prop.aut.Hn.tilde}, the automorphism $\Psi_M$ only depends on the image of $M$ in $\SL(2,\ZZ_n)$
because $T^m$ and $S^m$ only depend on $m$ modulo $n$. 

\subsubsection{Notation}
If $M \in \SL(2,\ZZ)$ and $x \in \widetilde{H}_n$ we will use the notation 
\begin{equation*}
M\triangleright x  \; := \;  \Psi_M(x).
\end{equation*}

\subsubsection{}
A nice feature of the action of $\SL(2,\ZZ)$ on $\widetilde{H}_n$  in \cref{prop.aut.Hn.tilde} 
 is the following: if we define $J_{(u,v)}:=T^uS^v \in \widetilde{H}_n$, then in $ \widetilde{H}_n/\langle \epsilon^{1/2} \rangle \cong
 \ZZ^2_n$ we have $M \triangleright J_{(u,v)} = J_{(u,v)M^t}$.

\subsection{Extending representations of $H_n$ to representations of $\widetilde{H}_n$}
It is well-known that the $n$-dimensional irreducible representations of $H_n$ are determined up to isomorphism 
by their central characters; see, for example, \cite[\S 2]{gh-heis}. 

For a moment, let $\rho:H_n \to \GL(n,\CC)$ be any irreducible representation of $H_n$ on $\CC^n$.
Because $\epsilon$ has order $n$ it acts on $\CC^n$ as multiplication
by an $n^{\rm th}$ root of unity, $\zeta$ say.

When $n$ is even $\rho$ can be extended to a representation $\widetilde{\rho}$ of $\widetilde{H}_n$  
by having $\epsilon^{1/2}$ act by a chosen square root of $\zeta$.
By \cite[Prop.~5.1]{fh-rt}, for example, there are exactly two irreducible $\widetilde{H}_n$-representations $\widetilde{\rho}$ such that
$\widetilde{\rho}\,|_{H_n} \cong \rho$, 
namely  $\widetilde{\rho}$ and   $\widetilde{\rho} \otimes \varepsilon$,
where $\varepsilon$ is the non-trivial 1-dimensional character of $\widetilde{H}_n$ that sends $H_n$ to 1. 
In particular, such a representation of $\widetilde{H}_n$ is also determined by its central character, i.e., 
the scalar by which the generator $\epsilon^{1/2}$ of the center acts. The next statement follows.

\begin{proposition}
\label{prop.rho.V}
Let $\rho:H_n \to \GL(n,\bC)$ be any $n$-dimensional irreducible representation of $H_n$ on which $\epsilon$
acts as multiplication by $\zeta$. Fix a square root, $\zeta^{1/2}$, of $\zeta$.
Extend $\rho$ to a group homomorphism $\widetilde{\rho}:\widetilde{H}_n \to \GL(n,\bC)$ by declaring that
\begin{equation*}
\widetilde{\rho}\big(T^aS^b \epsilon^{c/2}\big) \; :=\; \rho(T)^a\rho(S)^b (\zeta^{1/2})^c.
\end{equation*} 
If $\phi \in \Aut(\widetilde{H}_n)$  acts as the identity on $\epsilon^{1/2}$, then the representations 
$\widetilde{\rho}\phi$ and $\widetilde{\rho}$ are isomorphic.
\end{proposition}

\begin{corollary}
\label{cor.psi}
Let $\rho:H_n \to \GL(n,\bC)$ be any $n$-dimensional irreducible representation of $H_n$ and $\widetilde{\rho}:H_n \to \GL(n,\bC)$
an extension of it to $\widetilde{H}_n$. For $M \in \SL(2,\ZZ)$, let $\Psi_M \in \Aut(\widetilde{H}_n)$ be the automorphism in 
\Cref{prop.aut.Hn.tilde}.  There is a set map---not a group homomorphism, in general--- 
\begin{equation}
\label{eq:defn.psi}
\psi: \SL(2,\ZZ)\to \GL(V)
\end{equation}
 such that 
\begin{equation}
\label{eq:13-}
\widetilde{\rho}(\Psi_M(x)) \;= \; \psi(M) \, \widetilde{\rho}( x) \, \psi(M)^{-1}
\end {equation}
for all $x \in \widetilde{H}_n$. 
\end{corollary}
\begin{proof}
This follows from the previous proposition because $\Psi_M(\epsilon^{1/2}) = \epsilon^{1/2}$.  
\end{proof}

\subsubsection{The notation $\rho$}
From now on we will use the notation $\rho:\widetilde{H}_n \to \GL(V)$ for the representation of $\widetilde{H}_n$ 
in \cref{prop.qnk.act}, and also for its linear extension $\CC \widetilde{H}_n \to \End_\CC(V)$ to the group algebra.
We will therefore write \Cref{eq:13-} as
\begin{equation}\label{eq:13}
\rho(  M\triangleright x) \; = \;  \psi(M)\, \rho(x) \, \psi(M)^{-1}.
 \end{equation}

\subsubsection{Remark}
\label{ssect.Z.vs.Zn.2}
In analogy with the remark in \cref{ssect.Z.vs.Zn}, the automorphism $\psi(M):V \to V$ depends only on the image of $M$ in $\SL(2,\ZZ_n)$.
For example, if the image of $M$ in $\SL(2,\ZZ_n)$ is the identity, then $\Psi_M=\id_{\widetilde{H}_n}$ so, in \Cref{cor.psi}, we can take
$\psi(M)=\id_V$.

\subsection{The algebra  $A_n$ and the action of $\SL(2,\bZ)$ on it} 
\label{re:realact}
Let
\begin{equation*}
A_n \; := \; \bC \widetilde{H}_n\otimes_{\bC\langle \epsilon^{1/2}\rangle} \bC \widetilde{H}_n,
\end{equation*}
the tensor square of the group algebra $ \bC \widetilde{H}_n$ over the group algebra of its center, $\langle \epsilon^{1/2}\rangle$.
The action of $\SL(2,\ZZ)$ on $\widetilde{H}_n$ extends to an action of $\SL(2,\ZZ)$ on
$\CC\widetilde{H}_n$ and, because elements of $\SL(2,\ZZ)$ act as the identity on the center of $\widetilde{H}_n$, there is an
induced action of  $\SL(2,\ZZ)$ on $A_n$, namely  $M  \triangleright (x \otimes y) =(M \triangleright x) \otimes (M \triangleright y)$.

\section{The main results}\label{se:main}

\subsection{Introduction}
Let $g,h \in \GL(V)$ be the linear operators
\begin{equation*}
g\cdot x_i \; :=\; \omega^i x_i \phantom{xxxxxxxxxxxxx} h\cdot x_i\;:=\; x_{i-1}
\end{equation*}
where $\omega:=e\big(\frac{1}{n}\big)$.
The map $S\mapsto g$, $T \mapsto h$, extends to an irreducible representation of $H_n$ on $V$ 
(essentially because if $m$ is an integer, then $g^m$ and $h^m$ depend only on $m$ modulo $n$)
 in which $\epsilon$ acts as multiplication by $\omega^{-1}$. 
For $(a,b)\in \bZ_n^2$, we define the linear operator
\begin{equation*}
  I_{a,b} \; := \; h^ag^b  \, :V \to V,
\end{equation*} 
and its ``universal'' version
\begin{equation*}
J_{a,b} \; := \;J_{(a,b)}  \, := \, T^aS^b \, \in  \,\bC \widetilde{H}_n.
\end{equation*} 

\subsubsection{}
We note that $M \triangleright J_{(a,b)}= c \, J_{(a,b)M^t}$ where $c$ is some element in the center of $\widetilde{H}_n$.
Thus, in $A_n$ we have the useful equality
\begin{equation}
\label{eq:important}
M \triangleright \big(J_{(a,b)} \otimes J_{(a,b)}^{-1} \big) \;=\; \, J_{(a,b)M^t}  \otimes  J_{(a,b)M^t}^{-1}.
\end{equation}
This is the reason we treat $J_{(a,b)} \otimes J_{(a,b)}^{-1}$ as an element in $A_n$ rather than
$\bC \widetilde{H}_n \otimes_\bC \bC\widetilde{H}_n$ when we define $L_k(z,\eta \, | \, \tau)$ in \Cref{eq:linbn} below.

\subsubsection{}
Let   $T_k(z,\eta \, | \, \tau):V^{\otimes 2} \to V^{\otimes 2}$ 
be the linear operator
\begin{equation}
\label{defn.T(-nz)}
  T_{k}(z,\eta \, | \, \tau) := \sum_{(u,v)\in \bZ_n^2}w_{(u,v)}(-nz,\eta \, | \, \tau)I_{-k'u,v}\otimes I_{-k'u,v}^{-1}
\end{equation}
where $ k'  := \text{the unique integer such that $n>k' \ge 1$ and  $kk'=1\;\mathrm{mod}\; n$}$.

\subsubsection{}
Let $P:V \otimes V \to V \otimes V$ be the linear map $P(x \otimes y)=y\otimes x$.

\begin{theorem}
 \cite[Prop.~3.5]{CKS4}
The images of $R_{n,k}(z,\eta\, | \, \tau)$ and $PT_{k}(z,\eta \, | \, \tau)$ are the same because
\begin{equation}\label{eq:tisr}
R_{n,k}(z,\eta\, | \, \tau)  \; =\; \tfrac{1}{n} e(-\tfrac{1}{2}n(n+1)z) \, P \,T_k(z,\eta \, | \, \tau).
\end{equation}
Hence 
\begin{equation*}
Q_{n,k}(\eta \, | \, \tau) \; = \; \frac{TV}{(\textnormal{the image of $ P\, T_{k}(\eta,\eta \, | \, \tau)$})} \, .
\end{equation*}
\end{theorem}

The  only difference between the operator $R_{n,k}(z,\eta\, | \, \tau)$ defined in  \cref{the.relns} and 
$\tfrac{1}{n} e(-\tfrac{1}{2}n(n+1)z)PT_{k}(z,\eta \, | \, \tau)$ defined via \cref{defn.T(-nz)}  is that they are 
the matrix representations of the same linear operator with respect to different bases for $\End_\CC(V)$. 
The advantage of $T_{k}(z,\eta \, | \, \tau)$ over $R_{n,k}(z,\eta \, | \, \tau)$ is that we can exploit the way in which the actions of 
$\widetilde{H}_n$ and $A_n$ on $V^{\otimes 2}$ interact with the actions of $\SL(2,\ZZ)$ on $\widetilde{H}_n$ and the functions $w_{(u,v)}(z,\eta \, | \, \tau)$.\footnote{We learned this from Richey and Tracy's paper \cite{ric-tra}.}
To determine how $Q_{n,k}(\eta \, | \, \tau)$ transforms when $\SL(2,\ZZ)$ acts on $(\eta \, |\, \tau)$ we will
determine how $T_{k}(z,\eta \, | \, \tau)$ transforms under the action of $\SL(2,\ZZ)$ on $(z, \eta \, |\, \tau) \in \CC \times \CC \times \HH$.

\subsection{$\SL(2,\ZZ)$-equivariance properties of $T_{k}(z,\eta \, | \, \tau)$ and $L_{k}(z,\eta \, | \, \tau)$}\label{subse:opval}

Following \cite[Thm.~3.2]{CKS4} and the discussion before it, we will determine how $T_{k}(z,\eta \, | \, \tau)$ transforms under 
the action of 
$\SL(2,\ZZ)$ by examining the action of $\SL(2,\ZZ)$ on the ``universal'' version of $T_k(z,\eta \, | \, \tau)$ which belongs to the algebra 
$A_n$. That universal version is
\begin{equation}\label{eq:linbn}
  L_k(z,\eta \, | \, \tau)  \; :=\;  \sum_{(u,v)\in \bZ_n^2}w_{(u,v)}(-nz,\eta \, | \, \tau)J_{-k'u,v}\otimes J_{-k'u,v}^{-1} \; \in \, A_n.  
\end{equation}
The next ``result'' formalizes the statement that  $L_k(z,\eta \, | \, \tau)$ is the universal version of $T_k(z,\eta \, | \, \tau)$.

\begin{lemma}
If $\rho: \bC\widetilde{H}_n\to \End(V)$ is the representation in \Cref{prop.qnk.act}, then
\begin{equation*}
T_k(z,\eta \, | \, \tau) \; = \;  (\rho \otimes \rho)\big(L_k(z,\eta \, | \, \tau)\big).
\end{equation*} 
\end{lemma}

The main result in this section, \Cref{th:lequiv}, determines the ``equivariance'' properties of the map
\begin{equation} 
\label{eq:l}
L_k:  \bC\times\bC\times \bH \, \longrightarrow \, A_n, \qquad (z,\eta \, | \, \tau) \, \mapsto  \, L_k(z,\eta \, | \, \tau) 
\end{equation}
with respect to the actions of $\SL(2,\ZZ)$ on its domain and co-domain. 
To state it we introduce the matrix
\begin{equation}
\label{eq:defn.matrix.D}
D \; :=\;   
  \begin{pmatrix}
    -k&0\\
    0&1
  \end{pmatrix} 
\,  \in \, \GL(2,\bR)
\end{equation}
and, for $M\in\SL(2,\bZ)$, define
\begin{equation*}
M' \; : = \;  D^{-1}M^{-t}D.
\end{equation*}
Note that $D^{-t}=D^{-1}$ and $M''=(M')'=M$.

\begin{theorem}\label{th:lequiv}
Fix the data $(n,k,\eta,\tau)$ and $M \in  \SL(2,\ZZ)$. There is a nowhere-vanishing holomorphic function $f(z)$ such that 
  \begin{equation}\label{eq:twols}
 L_k\big( M \triangleright (z,\eta \, | \, \tau)\big) \; = \;   f(z) \, M' \triangleright L_k(z,\eta \, | \, \tau ),
  \end{equation}
where $M' \triangleright$  denotes  the action of $M' \in \SL(2,\ZZ)$ on $A_n$. 
\end{theorem}
\begin{proof}
Note that $(-k'u,v) =  (u,v)D^{-1}$ for all $(u,v) \in \ZZ_n^2$.

Since $M''=M$,  \Cref{eq:twols}  is equivalent to the statement that there is a nowhere-vanishing holomorphic function 
$f(z)$ such that 
\begin{equation*}
L_k\big( M' \triangleright (z,\eta \, | \, \tau)\big) = f(z) \,  M \triangleright L_k(z,\eta \, | \, \tau).
\end{equation*} 
That is what we will prove.

 By the definition of $L_k(z,\eta \, | \, \tau)$ in \Cref{eq:linbn},  
  \begin{align*}
    M\triangleright L_k(z,\eta \, | \, \tau) & \; = \; 
    \sum_{u,v}w_{(u,v)}(-nz,\eta \, | \, \tau) \, J_{(u,v)D^{-1}M^t}\otimes J_{(u,v)D^{-1}M^t}^{-1}
    \\
                                     &\; = \; \sum_{u,v} w_{ (u,v)M^{-t}D}(-nz,\eta \, | \, \tau)\, J_{u,v}\otimes J_{u,v}^{-1}
    \\
                                     &\;=\;\sum_{u,v} w_{(u,v)D(D^{-1}M^{-t}D) }(-nz,\eta \, | \, \tau) \, J_{u,v}\otimes J_{u,v}^{-1}.
  \end{align*}
  (The first equality in the previous computation uses the fact that $L_k(z,\eta \, | \, \tau)$ is, by definition, an element of $A_n$,
  not an element of $\CC \widetilde{H}_n \otimes_\CC \CC \widetilde{H}_n $.)
  By \Cref{th:wuvs}, there is a nowhere-vanishing holomorphic function $g(z)$ such that this equals
  \begin{equation*}
  g(z) \,   \sum_{u,v} w_{(u,v)D}\left(D^{-1}M^{-t}D \triangleright(-nz,\eta \, | \, \tau)\right) J_{u,v}\otimes J_{u,v}^{-1},
  \end{equation*}
  which is, as we see after re-indexing the summation over $(u,v)$, the same thing  as
  \begin{equation*}
g(z) \,    \sum_{u,v} w_{(u,v)}\left(M' \triangleright(-nz,\eta \, | \, \tau)\right) J_{-k'u,v}\otimes J_{-k'u,v}^{-1}
   \; = \;
g(z) \,    L\left(M' \triangleright(z,\eta \, | \, \tau)\right).
  \end{equation*}
  This finishes the proof.
\end{proof}

If $M ={{a \; \,b} \choose {c \, \; d}} \in \SL(2,\bZ)$, then the left-hand side of \Cref{eq:twols} is
\begin{equation*}
 L_k\bigg( \frac{z}{c\tau+d}\, , \, \frac{\eta}{c\tau+d}  \; \bigg \vert \;  \frac{a\tau+b}{c\tau+d}  \bigg)  .
\end{equation*}

\subsection{R-matrix and elliptic-algebra equivariance, and proof of the main theorem}\label{subse:stat}

We now translate \Cref{eq:twols}, which is a statement about how $L_k(z,\eta \, | \, \tau)$  transforms under the 
action of $\SL(2,\ZZ)$, into a statement about how $R_{n,k}(z,\eta \, | \, \tau)$ transforms under the 
action of $\SL(2,\ZZ)$. To do that we make use of the set map $\psi:\SL(2,\ZZ) \to \GL(V)$  in \Cref{cor.psi}.

\begin{theorem}\label{th:requiv}
Fix the data $(n,k,\eta,\tau)$. 
Let $M \in \SL(2,\ZZ)$ and let $\psi(M')$ be the automorphism of $V$ identified in 
\Cref{cor.psi}.  There is a nowhere-vanishing holomorphic function $f(z)$ such that
\begin{equation*}
 R_{n,k}\big(M \triangleright (z,\eta \, | \, \tau) \big) \; = \; 
  f(z) \, \psi(M')^{\otimes 2}  \cdot R_{n,k}(z,\eta \, | \, \tau) \cdot  \big( \psi(M')^{\otimes 2}\big)^{-1}.
   \end{equation*}
\end{theorem}
\begin{proof}
Let $\rho: \bC\widetilde{H}_n\to \End_\CC(V)$ denote the representation in \Cref{prop.qnk.act}.

Since there is a nowhere-vanishing holomorphic function $f(z)$ such that 
\begin{equation*}
R_{n,k}(z, \eta \, | \, \tau) \; = \; f(z) \, PT_{k}(z,\eta \, | \, \tau) \;=\; f(z) \, P (\rho \otimes \rho) (L_k(z,\eta \, | \, \tau)),
\end{equation*}
it suffices to show that 
\begin{equation*}
(\rho \otimes \rho)\big(L_k(M \triangleright (z,\eta \, | \, \tau))\big) \;=\;
g(z) \, \psi(M')^{\otimes 2} \cdot  (\rho \otimes \rho)\big(L_k(z,\eta \, | \, \tau)\big)  \cdot  \big( \psi(M')^{\otimes 2}\big)^{-1}
\end{equation*}
for some nowhere-vanishing holomorphic function $g(z)$.  However, it follows from \Cref{th:lequiv} that there is a nowhere-vanishing 
holomorphic function $g(z)$ such that 
\begin{align*}
(\rho \otimes \rho)\big(L_k(M \triangleright (z,\eta \, | \, \tau))\big)  & \;  = \; g(z) \, 
(\rho \otimes \rho) \big(M' \triangleright   L_k(z ,\eta \, | \, \tau ) \big),
\end{align*}
which is equal to 
\begin{equation*}
g(z) \, \psi(M')^{\otimes 2}  \cdot (\rho \otimes \rho) \big( L_k (z ,\eta \, | \, \tau ) \big) \cdot  \big( \psi(M')^{\otimes 2}\big)^{-1}
\end{equation*}
because, by \Cref{eq:13}, $\rho(M' \triangleright x) = \psi(M') \rho(x) \psi(M')^{-1}$ for all $x \in \widetilde{H}_n$, 
and hence for all $x \in \CC \widetilde{H}_n$. The proof is complete.
\end{proof}

\begin{theorem}[\Cref{thm.main}]\label{th:qequiv}
Let $M \in \SL(2,\ZZ)$. The linear isomorphism $\psi(M'):V \to V$ identified in \Cref{cor.psi}
 extends to a graded $\CC$-algebra isomorphism
  \begin{equation}\label{eq:12}
 \psi(M'): Q_{n,k}(\eta \, | \, \tau) \, \stackrel{\sim}{\longrightarrow} \, Q_{n,k}\big(M \triangleright (\eta \, | \, \tau) \big).
  \end{equation}  
  In particular, if  $M ={{a \; \,b} \choose {c \, \; d}} \in \SL(2,\bZ)$, then  
\begin{equation}
\label{eq:Q.isom}
Q_{n,k}\left(  \frac{\eta}{c\tau+d}  \, \bigg \vert \,  \frac{a\tau+b}{c\tau+d}  \right)  \; \cong \; Q_{n,k}(\eta \, | \, \tau).
\end{equation}
\end{theorem}
\begin{proof}
By definition, $Q_{n,k}\big(M \triangleright (\eta \, | \, \tau) \big) =TV/I$ where $I$ is the ideal generated by the image of 
\begin{equation*}
R_{n,k}(M \triangleright (\eta,\eta \, | \, \tau)) \; := \; \lim_{z \to \eta} R_{n,k}(M \triangleright (z,\eta \, | \, \tau)) 
\end{equation*}
By \Cref{th:requiv}, this is the same as the image of 
$
 \psi(M')^{\otimes 2}  \cdot R_{n,k}(\eta,\eta \, | \, \tau). 
 $
The conclusion follows.
\end{proof}

\subsubsection{$\psi(M')$ is $\widetilde{H}_n$-equivariant if and only if $M$ is the identity}
If $\psi:R \to R'$ is an isomorphism of rings there is an induced isomorphism of automorphism groups
$\widehat{\psi}:\Aut(R) \to \Aut(R')$, $\a \mapsto \psi \a \psi^{-1}$. Since $\widetilde{H}_n$ acts as automorphisms of 
$Q_{n,k}(\eta \, | \, \tau)$ and $Q_{n,k}(\eta' \, | \, \tau')$  it is natural to ask whether $\widehat{\psi(M')}$ sends the copy of 
 $\widetilde{H}_n$ in $\Aut(Q_{n,k}(\eta \, | \, \tau))$ to the copy of 
 $\widetilde{H}_n$ in $\Aut(Q_{n,k}(\eta' \, | \, \tau'))$. It does: if $x \in \widetilde{H}_n$, then 
 $\psi(M') \rho(x) \psi(M')^{-1}= \rho(M' \triangleright x) $; i.e., $\widehat{\psi(M')}$
 restricts to an isomorphism between the two copies of $\widetilde{H}_n$ and that isomorphism is 
 $x \mapsto M' \triangleright x$. In particular, if $M$ is not the identity, then $\psi(M')$ is not $\widetilde{H}_n$-equivariant.

\subsubsection{}
The next result involves an equality, not just an isomorphism.

\begin{corollary}\label{cor:waspuz}
If $m \in \ZZ$, then
$
Q_{n,k}\big(\frac{\eta}{mn\tau+1}  \, \big \vert \, \frac{\tau}{mn\tau+1}\big)  \, = \, Q_{n,k}(\eta \, | \, \tau).
$
\end{corollary}
\begin{proof}
As we remarked in \Cref{ssect.Z.vs.Zn,ssect.Z.vs.Zn.2}, 
 the automorphisms $\Psi_M$ of $\widetilde{H}_n$  and $\psi(M)$ of $V$ defined for 
$M\in  \SL(2,\ZZ)$ only depend on the image of $M$ in $\SL(2,\ZZ_n)$; in particular, if  the image of $M'$ in $\SL(2,\ZZ_n)$
is the identity, then we can take $\psi(M')=\id_V$. If $M ={{\; 1 \;\;  \; 0} \choose {mn \,\; 1}}$, then the image of $M'$ in $\SL(2,\ZZ_n)$
is the identity so we can take $\psi(M')=\id_V$, whence \Cref{th:qequiv} gives the result.
\end{proof}

\subsection{The proof of \Cref{thm.main.2}}

That result is a consequence of the following lemma.

\begin{lemma}
\label{lem.var.hain}
Let $\Lambda_1$ and $\Lambda_2$ be lattices in $\bC$ and $f:\bC/\Lambda_1 \to \bC/\Lambda_2$ a morphism of algebraic
groups. 
\begin{enumerate}
\item\label{item.var.hain.ex}
There is $u \in \bC^\times$ such that $f(z+\Lambda_1) \, = \, uz + \Lambda_2$ for all $z \in \bC$.
\item\label{item.var.hain.unique}
If $f$  is an isomorphism of algebraic groups, then the $u$ in \cref{item.var.hain.ex} is unique.
\item\label{item.var.hain.sl}
Assume $\Lambda_1=\Lambda_{\tau_1}$ and $\Lambda_2=\Lambda_{\tau_2}$ for some $\tau_1,\tau_2 \in \bH$.
If $f$  is an isomorphism of algebraic groups, then there is a unique 
\begin{equation*}
\begin{pmatrix}  a    &  b  \\ c & d \end{pmatrix} \, \in \,\SL(2,\bZ)
\end{equation*}
such that $f(z+\Lambda_1) \, = \, \frac{z}{c\tau_1+d} + \Lambda_2$  for all $z \in \bC$ and $\tau_2=\frac{a\tau_1+b}{c\tau_1+d}$.  
\end{enumerate} 
\end{lemma}
\begin{proof} 
\cref{item.var.hain.ex} This is \cite[Lem.~1.9]{Hain11}.

\cref{item.var.hain.unique}
If $u_1z + \Lambda_2=u_2z + \Lambda_2$ for all $z \in \bC$, then $(u_1-u_2)z  \in \Lambda_2$ for all $z \in \bC$, which 
implies that $u_1-u_2=0$. Hence the $u$ in \cref{item.var.hain.ex} is unique.

\cref{item.var.hain.sl}
(Existence.) 
Since $f(z+\Lambda_1) \; = \; uz + \Lambda_2$ for all $z \in \bC$, $u\Lambda_1=\Lambda_2$. 
Hence $\tau_2=u(a\tau_1+b)$ and $1=u(c\tau_1+d)$ for some $a,b,c,d\in \bZ$. Since $u\Lambda_1=\Lambda_2$,
$\{a\tau_1+b,c\tau_1+d\}$ is a $\bZ$-basis for $\Lambda_{\tau_1}$. Hence ${a \; b \choose c \; d} \in \GL(2,\bZ)$. 
Since $\tau_2=\frac {a\tau_1+b}{c\tau_1 + d}$, the imaginary part of $\tau_2$  is
\begin{equation*}
\Im(\tau_2) \;=\; \det \begin{pmatrix} a & b \\ c & d \end{pmatrix} |c\tau_1 +d|^{-2} \Im(\tau_1).
\end{equation*}
Since the imaginary parts of $\tau_1$ and $\tau_2$ are positive,  
$\det\!{a \; b \choose c \; d}>0$. Hence ${a \; b \choose c \; d}$  is in $\SL(2,\bZ)$. 

(Uniqueness.)  Since $\tau_1\notin\bQ$, the integers $c$ and $d$ in the equality $1=u(c\tau_1+d)$ are unique.  Likewise, since $\tau_1 \notin\bQ$, the integers $a$ and $b$ in the equality $a\tau_1+b=u^{-1}\tau_2$ are unique.
\end{proof}

\begin{corollary}[\Cref{thm.main.2}]
\label{cor.qnk.not}
If $f:E_{\tau} \to E_{\tau'}$ is an isomorphism of algebraic groups such that $f(\eta+\Lambda_{\tau})=\eta'+\Lambda_{\tau'}$, then 
\begin{equation*}
Q_{n,k}(\eta\,|\,\tau) \;\cong\; Q_{n,k}(\eta'\,|\,\tau').
\end{equation*}
\end{corollary}
\begin{proof}
By \cref{lem.var.hain}\cref{item.var.hain.sl}, there is $M={a \; b \choose c \; d}\in\SL(2,\bZ)$ such that  $f(z+\Lambda_\tau) \, = \, \frac{z}{c\tau+d} + \Lambda_{\tau'}$  for all $z \in \bC$ and $\tau'=\frac{a\tau+b}{c\tau+d}$. In particular, $\eta'+\Lambda_{\tau'}=\frac{\eta}{c\tau+d} + \Lambda_{\tau'}$. By \Cref{th:qequiv},
\begin{equation*}
Q_{n,k}(\eta\,|\,\tau)\;\cong\;Q_{n,k}\big(\tfrac{\eta}{c\tau+d}\,\big\vert\,\tfrac{a\tau+b}{c\tau+d}\big)\;=\;Q_{n,k}(\eta'\,|\,\tau')
\end{equation*}
where the last equality holds since $\frac{\eta}{c\tau+d}$ and $\eta'$ have the same image in $E_{\tau'}$.
\end{proof}

\subsection{Modularity}\label{subse:mod}
The action of the modular group $\SL(2,\ZZ)$ on $\CC \times \HH$ is the starting point for the theory of 
modular forms and modular functions. 
For example, 
(see, e.g.,  \cite[Defn.~1.1.1]{DS05} and \cite[\S I.3]{silv2}) a meromorphic function $f:\bH \to \CC$ is called a  
 {\sf weakly modular function of weight $k$} if 
\begin{equation}\label{eq:modf}
f\left(\frac{a\tau+b}{c\tau+d}\right) \;=\;  (c\tau+d)^{k}   f(\tau) 
\end{equation}
for all $\tau \in \HH$ and all 
\begin{equation*}  
  \begin{pmatrix}
    a&b\\
    c&d
  \end{pmatrix}
 \, \in \, \SL(2,\bZ). 
\end{equation*}
We will now restate \Cref{eq:modf} in a way that resembles the isomorphism in \Cref{eq:Q.isom}.

First, we associate to a function $f:\HH \to \CC$ its (backward) graph
\begin{equation*}
\G_f \; :=\; \{(f(\tau) \, | \, \tau) \; | \; \tau \in \HH\} \, \subseteq \CC \times \HH.
\end{equation*}
Subsets of a given set can be viewed as $\{0,1\}$-valued functions on that set. Here such a procedure  produces the function
\begin{equation*}
F:\CC \times \HH \to \{0,1\}, \qquad F(z \, | \, \tau) \; :=\; 
\begin{cases} 
1 & \text{if $f(\tau)=z$,}
\\
0 & \text{if $f(\tau) \ne z$.}
\end{cases}
\end{equation*}
With this notation,  $f$ satisfies \Cref{eq:modf} if and only if
\begin{equation*}
F \bigg( (c\tau+d)^{k} z \; \bigg \vert \; \frac{a\tau+b}{c\tau+d}  \bigg) \;=\; F(z \, | \, \tau).
\end{equation*}
In particular, $f(\tau)$ is a  weakly modular function of weight $-1$ if and only if the $\{0,1\}$-valued function  
$F(z \, | \, \tau)$ satisfies the identity
\begin{equation*}
F \bigg(\frac{z}{c\tau+d} \; \bigg \vert \;   \frac{a\tau+b}{c\tau+d}  \bigg) \;=\; F(z \, | \, \tau),
\end{equation*}
which resembles the isomorphism in \Cref{eq:Q.isom}.

If we express this identity by saying that $F$ is a  ``weakly modular function  of 
weight $-1$  taking values in $\{0,1\}$'', and extend this terminology in the obvious way, then 
\Cref{th:qequiv} says that $Q_{n,k}$ is a ``weakly modular function  of 
weight $-1$ taking values in the category of graded $\CC$-algebras''.



\def\cprime{$'$}
\providecommand{\bysame}{\leavevmode\hbox to3em{\hrulefill}\thinspace}
\providecommand{\MR}{\relax\ifhmode\unskip\space\fi MR }
\providecommand{\MRhref}[2]{%
  \href{http://www.ams.org/mathscinet-getitem?mr=#1}{#2}
}
\providecommand{\href}[2]{#2}

\end{document}